\documentclass[a4paper,reqno,11pt]{amsart}
\usepackage{amsmath}
\usepackage{amssymb}
\usepackage{amsthm}
\usepackage{latexsym}
\usepackage[all]{xy}
\usepackage{amsfonts}
\usepackage{mathrsfs}
\usepackage{graphicx,color}
\usepackage{xcolor}
\usepackage{tikz}
\usepackage{graphics}

\usepackage{anysize,hyperref}

\newcommand{\cal}{\mathcal}

\newtheorem{theorem}{Theorem}[section]
\newtheorem{lemma}[theorem]{Lemma}
\newtheorem{corollary}[theorem]{Corollary}
\newtheorem{proposition}[theorem]{Proposition}
\theoremstyle{definition}
\newtheorem{definition}[theorem]{Definition}
\newtheorem{definition-proposition}[theorem]{Definition-Proposition}
\newtheorem{remark}[theorem]{Remark}
\newtheorem{example}[theorem]{Example}

\def\C{\mathcal{C}}
\def\D{\mathcal{D}}
\def\T{\mathcal{T}}

\def\X{\mathscr{X}}
\def\Y{\mathscr{Y}}

\def\I{\mathcal {I}}

\def \text{\mbox}

\providecommand{\add}{\mathop{\rm add}\nolimits}%
\providecommand{\Ext}{\mathop{\rm Ext}\nolimits}%
\providecommand{\Hom}{\mathop{\rm Hom}\nolimits}%
\renewcommand{\mod}{\mathop{\rm mod}\nolimits}%
\providecommand{\rep}{\mathop{\rm rep}\nolimits}%

\def\k{\mathbf{k}}

\def\pp{\mathfrak{p}}
\def\qq{\mathfrak{q}}
\def\qqq{\mathfrak{qq}}
\def\ii{\mathfrak{i}}
\def\aa{\mathfrak{a}}
\def\jj{\mathfrak{j}}
\def\ll{\mathfrak{l}}
\def\rr{\mathfrak{r}}
\def\TT{\mathbb{P}}
\def\bTT{\overline{\TT}}

\def\BB{\mathcal{B}_{\infty}}
\def\XX{\widetilde{\X}}
\def\YY{\widetilde{\Y}}
\def\CC{\mathscr{C}}

\providecommand{\cona}{\mathop{\rm (Pt)}\nolimits}%
\providecommand{\conb}{\mathop{\rm (Pt2)}\nolimits}%
\providecommand{\conbp}{\mathop{\rm (Pt2')}\nolimits}%
\providecommand{\conc}{\mathop{\rm (Pt3)}\nolimits}%
\providecommand{\nc}{\mathop{\rm nc}\nolimits}%

\begin{document}

\title{Cotorsion pairs in cluster categories of type $A_{\infty}^{\infty}$}

\author[Chang]{Huimin Chang}
\address{
Department of Mathematical Sciences
Tsinghua University
100084 Beijing
China
}
\email{chm14@mails.tsinghua.edu.cn}

\author[Zhou]{Yu Zhou}
\address{
Department of Mathematical Sciences
Norwegian University of Science and Technology
7491 Trondheim
Norway
}
\email{yu.zhou@ntnu.no}

\author[Zhu]{Bin Zhu}
\address{
Department of Mathematical Sciences
Tsinghua University
100084 Beijing
China
}
\email{bzhu@math.tsinghua.edu.cn}

\begin{abstract}

In this paper, we give a complete classification of cotorsion pairs in a cluster category $\CC$ of type $A^\infty_\infty$ via certain configurations of arcs, called $\tau$-compact Ptolemy diagrams, in an infinite strip with marked points. As applications, we classify $t$-structures and functorially finite rigid subcategories in $\CC$, respectively. We also deduce Liu-Paquette's classification of cluster tilting categories of $\CC$ and Ng's classification of torsion pairs in the cluster category of type $A_\infty$.

\end{abstract}



\thanks{This work was supported by the NSF of China (Grants No.\;11671221). The second author was partially supported by FRINAT grant number 231000, from the Norwegian Research Council. }

\maketitle

\section{Introduction}


Torsion theory is a fundamental and central topic in the representation theory of algebras. Torsion pairs for abelian categories, introduced by Dickson \cite{D}, are intimately related to tilting theory. The ideal of torsion theory for a triangulated category was introduced by Iyama and Yoshino \cite{IY} to study cluster tilting subcategories in a triangulated category.

Cluster categories, constructed by Buan, Marsh, Reineke, Reiten and Todorov \cite{BMRRT} (also by Caldero, Chapoton and Schiffler \cite{CCS} for type $A_n$), give a categorical model for Fomin and Zelevinsky's cluster algebras. The cluster tilting subcategories of the cluster category correspond to the clusters of the cluster algebra and their mutations are compatible. Further, the torsion pairs in the cluster category correspond to certain pairs of cluster subalgebras of the cluster algebra (cf. \cite{CZ2,CZZ}).

Cotorsion pairs in a triangulated category were used by Nakaoka \cite{N} to unify the abelian structures arising from $t$-structures and from cluster tilting subcategories. Torsion pairs and cotorsion pairs in a triangulated category can be transformed into each other by shifting the torsion-free parts. Hence to classify torsion pairs is equivalent to classifying cotorsion pairs in a triangulated category. Note that this is not true for abelian categories.

Torsion/cotorsion pairs have been classified for many cluster categories (or more generally, 2-Calabi-Yau categories with maximal rigid subcategories):
\begin{enumerate}
	\item Ng \cite{Ng} classified torsion pairs in the cluster category of type $A_\infty$ (introduced in \cite{HJ2}) via certain configurations of arcs of the infinity-gon.
	\item Holm, J{\o}rgensen and Rubey \cite{HJR1,HJR2,HJR3} classified torsion pairs in the cluster category of type $A_{n}$, in the cluster tube and in the cluster category of type $D_n$ via Ptolemy diagrams of a regular $(n+3)$-gon, periodic Ptolemy diagrams of the infinity-gon and Ptolemy diagrams of a regular $2n$-gon, respectively.
	\item Zhang, Zhou and Zhu \cite{ZZZ} classified cotorsion pairs in the cluster category of an unpunctured marked surface via paintings of the surface.
	\item Zhou and Zhu \cite{ZZ2} classified torsion pairs in an arbitrary $2$-Calabi-Yau triangulated category with cluster tilting objects via decompositions of the triangulated category w.r.t. rigid objects.
	\item Chang and Zhu \cite{CZ} classified torsion pairs in finite $2$-Calabi-Yau triangulated categories with maximal rigid objects via periodic Ptolemy diagrams of a regular polygon.
\end{enumerate}

Notice that the works above only deal with 2-Calabi-Yau categories having cluster tilting subcategories or maximal rigid subcategories, which contain finitely many indecomposable objects except Ng's work. Recently, Liu and Paquette \cite{LP} introduced another 2-Calabi-Yau category, the cluster category $\CC$ of type $A_\infty^\infty$, which admits cluster categories having infinitely many indecomposable objects. They gave a geometric realization of $\CC$, via an infinite strip with marked points $\BB$ in the plane. Parameterizing the indecomposable objects in $\CC$ by the arcs in $\BB$, they showed that there is a bijection between the cluster tilting subcategories of $\CC$ and the compact triangulations of $\BB$.

In this paper, we introduce the definition of $\tau$-compact Ptolemy diagrams of $\BB$, which can be regarded as a generalization of compact triangulations of $\BB$. We show that there is a bijection between the cotorsion pairs in $\CC$ and the $\tau$-compact Ptolemy diagrams of $\BB$. A criterion for a Ptolemy diagram to be $\tau$-compact is also given. As applications, we get geometric descriptions of $t$-structures and functorially finite rigid subcategories in $\CC$, respectively. We also deduce Liu-Paquette's classification of cluster tilting categories of $\CC$ and Ng's classification of torsion pairs in the cluster category of type $A_\infty$.

The paper is organized as follows. In Section 2, we review background materials concerning cotorsion pairs in a triangulated category and cluster categories of type $A_\infty^\infty$. In Section 3, we introduce the notion of $\tau$-compact Ptolemy diagrams of $\BB$ and give a criterion for a Ptolemy diagram to be $\tau$-compact. Section 4 is devoted to proving the main result (Theorem~\ref{maintheorem}) of this paper. Many applications are given in the last section.

\subsection*{Conventions} Throughout this paper, $\k$ stands for an algebraically closed field, and all categories are assumed to be Hom-finite, Krull-Schmidt and $\k$-linear. Any subcategory of a category is assumed to be full and closed under taking isomorphisms, finite direct sums and direct summands. For a subcategory $\X$ of a category $\D$, we denote by $\X^\perp$ (resp. $^\perp\X$) the subcategory whose objects are $M\in\D$ satisfying $\Hom_{\D}(X,M)=0$ (resp. $\Hom_{\D}(M,X)=0$) for any $X\in \X$. For two subcategories $\X,\Y$ of $\D$, $\Hom(\X,\Y)=0$ means $\Hom_{\D}(X,Y)=0$ for any $X\in\X$ and any $Y\in\Y$. For two subcategories $\X,\Y$ of a triangulated category $\D$, denote by $\X\ast\Y$ the subcategory of $\D$ whose objects are $M$ which fits into a triangle
\[X\rightarrow M\rightarrow Y\rightarrow X[1]\]
with $X\in\X$ and $Y\in\Y$. In a triangulated category, we use $\Ext^1(X,Y)$ to denote $\Hom(X,Y[1])$, where $[1]$ is the shift functor of the triangulated category.

\section{Preliminaries}

\subsection{Cotorsion pairs in triangulated categories}

We recall some (equivalent) definitions and results concerning cotorsion pairs in a triangulated category.

\begin{definition}\label{def:torsion} Let $\X,\Y$ be subcategories of a triangulated category $\D$.
\begin{enumerate}
\item The pair
$(\X,\Y)$ is called a torsion pair \cite{IY} if
\[\Hom_{\D}(\X,\Y)=0,\ \text{ and }\ \D=\X\ast\Y.\]
\item The pair $(\X,\Y)$ is called a cotorsion pair \cite{N} if
\[\Ext^1_{\D}(\X,\Y)=0,\ \text{ and }\ \D=\X\ast\Y[1].\]
The subcategory $\I:=\X\bigcap \Y$ is called the core \cite{ZZ3} of $(\X,\Y)$.
\item The pair $(\X,\Y)$ is called a t-structure \cite{BBD} if and only if it is a cotorsion pair and $\X$ is closed under shift (or equivalently, $\Y$ is closed under [-1]). The subcategory $\X[-1]\cap\Y[1]$ is called the heart of $(\X,\Y)$.
\item The subcategory $\X$ is called rigid if $\Ext^1(\X,\X)=0$.
\item The subcategory $\X$ is called a cluster tilting subcategory \cite{BMRRT,IY,KR1,KZ} if it satisfies the following:
\begin{enumerate}
\item $\X$ is contravariantly finite \cite{AS}, i.e. for any $M\in\D$, there is a morphism $X\rightarrow M$ such that any morphism $X'\rightarrow M$ with $X'\in\X$ factors through it.
\item $\X$ is covariantly finite \cite{AS}, i.e. for any  $M\in\D$, there is a morphism $M\rightarrow X$ such that any morphism $M\rightarrow X'$ with $X'\in\X$ factors through it.
\item $X\in\X$ if and only if $\Ext^1(X,X')=0$ for any $X'\in\X$ if and only if $\Ext^1(X',X)=0$ for any $X'\in\X$.
\end{enumerate}
\item The subcategory $\X$ is called functorially finite if it is contravariantly finite and covariantly finite.
\end{enumerate}

\end{definition}

\begin{proposition}[\cite{IY,ZZ3}]\label{prop:def}
Let $\X,\Y$ be subcategories of a triangulated category $\D$.
\begin{enumerate}
\item The pair
$(\X,\Y)$ is a torsion pair if and only if the following hold.
\begin{enumerate}
	\item $\X^\bot=\Y$;
	\item $^{\bot}\Y=\X$;
	\item $\X$ is contravariantly finite or $\Y$ is covariantly finite.
\end{enumerate}
\item The pair $(\X,\Y)$ is a cotorsion pair if and only if $(\X, \Y [1])$ is a torsion pair.
\item The pair $(\X, \Y)$ is a t-structure if and only if it is a cotorsion pair whose core is 0.
\item The subcategory $\X$ is functorially finite rigid if $(\X,\X^\bot)$ and $(^\bot\X,\X)$ are torsion pairs.
\item The subcategory $\X$ is cluster tilting if and only if $(\X,\X)$ is a cotorsion pair.
\end{enumerate}

\end{proposition}


\begin{definition}
A triangulated category $\D$ is called 2-Calabi-Yau (shortly 2-CY) provided there is a functorially isomorphism
\[\Hom_{\D}(X,Y)\simeq D\Hom_{\D}(Y, X[2]),\]
for any $X, Y\in \D$, where $D=\Hom_{\k}(-,\k)$.
\end{definition}

\subsection{Geometric description of cluster category of type \texorpdfstring{$A_{\infty}^{\infty}$}{A double infinity}}

In this subsection, we recall from \cite{LP} a geometric description of a cluster category of type $A^\infty_\infty$.

Let $Q$ be a quiver of type $A_{\infty}^{\infty}$ without infinite path, and $\rep(Q)$ the category of finite dimensional $\k$-linear representations of $Q$. Let $D^{b}(\rep(Q))$ be the bounded derived category of $\rep(Q)$ with shift functor [1] and the Auslander-Reiten translation $\tau$. The cluster category $\CC$ is defined to be the orbit category
\[\CC=\CC(Q)=D^{b}(\mathrm{rep}(Q))/\tau^{-1}[1].\]
By \cite{K}, $\CC$ is a Hom-finite Krull-Schmidt 2-Calabi-Yau triangulated $\k$-category.

Following \cite{LP}, denote by $\BB$ the infinite strip in the plane of the marked points $(x,y)$ with $0\leq y\leq 1$. The points $\mathfrak{l}_{i}=(i,1)$, $i\in\mathbb{Z}$, are called {\em  upper marked points}, and the points $\mathfrak{r}_{i}=(-i,0)$, $i\in\mathbb{Z}$, are called {\em lower marked points}. An upper or lower marked point will be simply called a {\em marked point}. For any two distinct marked points $\mathfrak{p}$ and $\mathfrak{q}$ in $\BB$, there exists a unique (up to isotopy) simple curve in $\BB$ joining them, which is written as $[\mathfrak{p}, \mathfrak{q}]$ or $[\mathfrak{q}, \mathfrak{p}]$. A simple curve $[\mathfrak{p}, \mathfrak{q}]$ in $\BB$ is called an {\em edge} if $\{\mathfrak{p}, \mathfrak{q}\}=\{\mathfrak{l}_{i},\mathfrak{l}_{i+1}\}$ or $\{\mathfrak{p}, \mathfrak{q}\}=\{\mathfrak{r}_{i},\mathfrak{r}_{i+1}\}$ for some $i\in\mathbb{Z}$, and otherwise, an {\em arc}. An arc in $\BB$ joining two upper marked points is called an {\em upper arc}; an arc in $\BB$ joining two lower marked points is called a {\em lower arc}; and an arc joining one upper marked point and one lower marked point is called a {\em connecting arc}. See Figure~\ref{fig:a}.

\begin{figure}[ht]\centering
\begin{tikzpicture}[scale=.6]
\draw[very thick] (-7,0)--(8,0);
\draw[very thick] (-7,5)--(8,5);
\foreach \x in {-6,-5,-4,...,7}
\draw (\x,0)node{$\bullet$} (\x,5)node{$\bullet$}
(\x,5)node[above]{$\ll_{\x}$};
\foreach \x in {-7,-6,-5,...,6}
\draw (-\x,0)node[below]{$\rr_{\x}$};
\draw(-3,5).. controls +(-60:1) and +(-120:1)..(-1,5);
\draw(2,5).. controls +(-140:2) and +(80:1)..(-1,0);
\draw(-1,0).. controls +(45:3) and +(135:3)..(5,0);
\end{tikzpicture}
\caption{Marked points and arcs in $\cal B_{\infty}$}
\label{fig:a}
\end{figure}
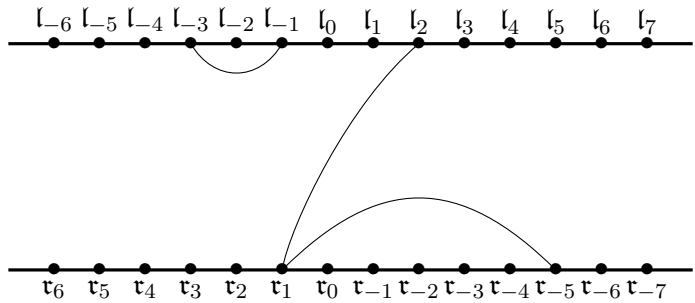

There is a translation $\tau$ on the set of arcs in $\BB$ given by
\[\tau [\pp,\qq]=[\tau\pp, \tau\qq]\]
where the translation $\tau$ acting on a marked point is given by $\tau\ll_i=\ll_{i+1}$ and $\tau\rr_i=\rr_{i+1}$ for any $i\in\mathbb{Z}$.

Let $u,v$ be arcs in $\cal B_{\infty}$. One says that u {\em crosses} v, or $(u,v)$ is a {\em crossing pair}, if every curve isotopic to $u$ crosses each of the curves isotopic to $v$. By definition, an arc does not cross itself, two crossing arcs do not share a common endpoint, and an upper arc does not cross any lower arc. The following lemma from \cite{LP} gives an explicit criterion for two arcs in $\BB$ to form a crossing pair, which will be frequently used without a reference.

\begin{lemma}[Lemma 4.2 in \cite{LP}]
Let $(u,v)$ be a crossing pair of arcs in $\BB$.
\begin{itemize}
   \item [(1)] If $u=[\mathfrak{l}_{i}, \mathfrak{l}_{j}]$ with $i<j$, then $v=[\mathfrak{l}_{p}, \mathfrak{r}_{q}]$ with $i<p<j$; or $v=[\mathfrak{l}_{p}, \mathfrak{l}_{q}]$ with $i<p<j<q$ or $p<i<q<j$.
   \item [(2)] If $u=[\mathfrak{r}_{i}, \mathfrak{r}_{j}]$ with $i>j$, then $v=[\mathfrak{l}_{p}, \mathfrak{r}_{q}]$ with $i>q>j$; or $v=[\mathfrak{r}_{p}, \mathfrak{r}_{q}]$ with $i>p>j>q$ or $p>i>q>j$.
   \item [(3)] If $u=[\mathfrak{l}_{i}, \mathfrak{r}_{j}]$, then $v=[\mathfrak{l}_{p}, \mathfrak{l}_{q}]$ with $p<i<q$; or $v=[\mathfrak{r}_{p}, \mathfrak{r}_{q}]$ with $p>j>q$; or  $v=[\mathfrak{l}_{p}, \mathfrak{r}_{q}]$ with $i>p$ and $j>q$ or $i<p$ and $j<q$.
\end{itemize}
\end{lemma}

We illustrate in Figure~\ref{fig:b} the different cases in the above lemma.

\begin{figure}[ht]\centering

\begin{tikzpicture}[xscale=.35,yscale=.4]
\draw[very thick] (-7,0)--(8,0);
\draw[very thick] (-7,5)--(8,5);
\draw (-1,5)node{$\bullet$}node[above]{$\ll_p$} (-4,5)node{$\bullet$}node[above]{$\ll_i$}
(5,5)node{$\bullet$}node[above]{$\ll_q$}
(2,5)node{$\bullet$}node[above]{$\ll_j$}
(2,0)node[below]{\color{white} $k$};
\draw(-4,5).. controls +(-30:3) and +(-150:3)..(2,5);
\draw(-1,5).. controls +(-30:3) and +(-150:3)..(5,5);
\end{tikzpicture}
\begin{tikzpicture}[xscale=.35,yscale=-.4]
\draw[very thick] (-7,0)--(8,0);
\draw[very thick] (-7,5)--(8,5);
\draw (-1,5)node{$\bullet$}node[below]{$\rr_p$} (-4,5)node{$\bullet$}node[below]{$\rr_i$}
(5,5)node{$\bullet$}node[below]{$\rr_q$}
(2,5)node{$\bullet$}node[below]{$\rr_j$}
(2,0)node[above]{\color{white} $k$};
\draw(-4,5).. controls +(-30:3) and +(-150:3)..(2,5);
\draw(-1,5).. controls +(-30:3) and +(-150:3)..(5,5);
\end{tikzpicture}

\begin{tikzpicture}[xscale=.35,yscale=.4]
\draw[very thick] (-7,0)--(8,0);
\draw[very thick] (-7,5)--(8,5);
\draw (0,5)node{$\bullet$}node[above]{$\ll_p$} (-3,5)node{$\bullet$}node[above]{$\ll_i$}
(-2,0)node{$\bullet$}node[below]{$\rr_q$}
(3,5)node{$\bullet$}node[above]{$\ll_j$};
\draw(-3,5).. controls +(-30:3) and +(-150:3)..(3,5);
\draw(0,5).. controls +(-170:2) and +(80:1)..(-2,0);
\end{tikzpicture}
\begin{tikzpicture}[xscale=.35,yscale=-.4]
\draw[very thick] (-7,0)--(8,0);
\draw[very thick] (-7,5)--(8,5);
\draw (0,5)node{$\bullet$}node[below]{$\rr_q$} (-3,5)node{$\bullet$}node[below]{$\rr_i$}
(-2,0)node{$\bullet$}node[above]{$\ll_p$}
(3,5)node{$\bullet$}node[below]{$\rr_j$};
\draw(-3,5).. controls +(-30:3) and +(-150:3)..(3,5);
\draw(0,5).. controls +(-170:2) and +(80:1)..(-2,0);
\end{tikzpicture}

\begin{tikzpicture}[xscale=.35,yscale=.4]
\draw[very thick] (-7,0)--(8,0);
\draw[very thick] (-7,5)--(8,5);
\draw (1,0)node{$\bullet$}node[below]{$\rr_j$} (-3,5)node{$\bullet$}node[above]{$\ll_i$}
(-2,0)node{$\bullet$}node[below]{$\rr_q$}
(3,5)node{$\bullet$}node[above]{$\ll_p$};
\draw (-3,5).. controls +(-70:2) and +(120:1) .. (1,0);
\draw (-2,0).. controls +(60:1.3) and +(-120:3) .. (3,5);
\end{tikzpicture}

\caption{Crossing arcs in $\cal B_{\infty}$}
\label{fig:b}
\end{figure}
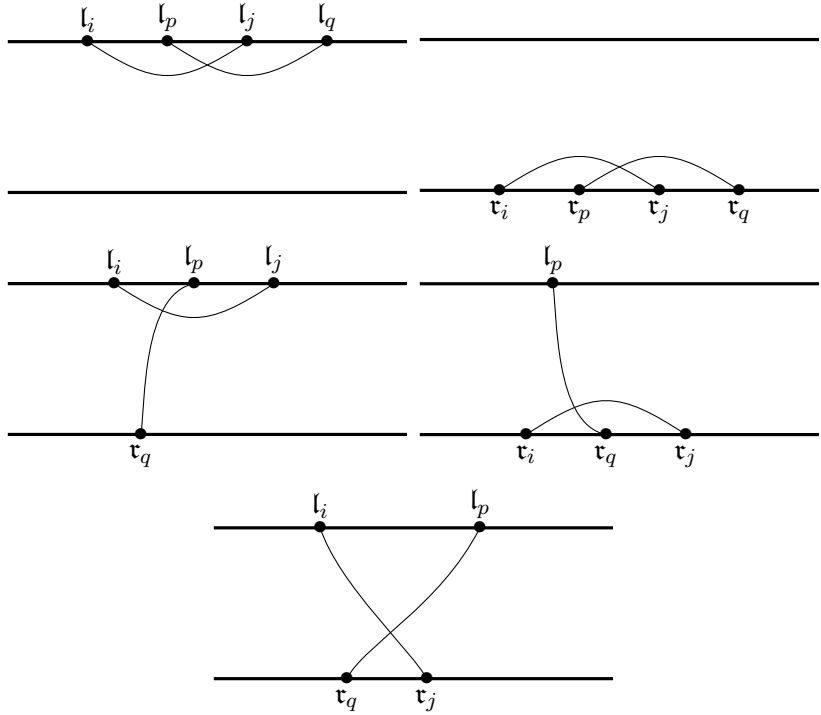


The infinite strip $\BB$ with marked points gives a geometric model for the cluster category $\CC$ in the following sense.

\begin{proposition}[Theorem 5.3 and Corollary 5.4 in \cite{LP}]\label{lem1}
There is a bijection from the set of (isoclasses of) indecomposable objects in $\CC$ to the set of (isotopy classes of) arcs in $\BB$. Moreover, let $u,v$ be arcs in $\BB$ and $M_{u}, M_{v}$ the corresponding indecomposable objects in $\CC$. Then
\begin{enumerate}
	\item $(u, v)$ is a crossing pair if and only if $\Ext^1_\CC(M_{u}, M_{v})\neq 0$; and
	\item $M_u[1]=M_{\tau u}$.
\end{enumerate}
\end{proposition}

The bijection in the above proposition induces a bijection between the subcategories of $\CC$ and the sets of arcs in $\cal B_{\infty}$. For a subcategory $\X$ of $\mathscr{C}$, we denote the corresponding set of arcs in $\BB$ by $\XX$.

We shall use the following notions, which is essentially from \cite{LP}.

\begin{definition}
Let $\TT$ be a set of arcs in $\BB$.
\begin{enumerate}
\item A marked point $\pp$ is called {\em upper left $\TT$-bounded} if there is an integer $j$ such that $[\pp,\ll_i]\notin\TT$ for any $i<j$.
\item A marked point $\pp$ is called {\em upper right $\TT$-bounded} if there is an integer $j$ such that $[\pp,\ll_i]\notin\TT$ for any $i>j$.
\item A marked point $\pp$ is called {\em lower left $\TT$-bounded} if there is an integer $j$ such that $[\pp,\rr_i]\notin\TT$ for any $i>j$.
\item A marked point $\pp$ is called {\em lower right $\TT$-bounded} if there is an integer $j$ such that $[\pp,\rr_i]\notin\TT$ for any $i<j$.
\end{enumerate}
\end{definition}

\section{Compact Ptolemy diagrams of \texorpdfstring{$\BB$}{B infinity}}

In this section, we introduce and study $\tau$-compact Ptolemy diagrams of $\BB$, which will be a geometric model for cotorsion pairs in $\CC$ in the next section.

For any marked point $\pp$ in $\BB$, set
\[[\pp,-]=\{[\pp,\qq]\mid \qq\text{ is a marked point in }\BB \}.\]
We define a linear order on $[\pp,-]$, that $[\pp,\ii]>_\pp[\pp,\jj]$ if and only if $[\pp,\jj]$ follows $[\pp,\ii]$ in the clockwise orientation. More explicitly,
\begin{itemize}
	\item when $\pp$ is an upper marked point, say $\pp=\ll_p$, we have \[[\ll_p,\ll_f]>_\pp[\ll_p,\ll_e]>_\pp[\ll_p,\rr_d]>_\pp[\ll_p,\rr_c]>_\pp[\ll_p,\ll_b]>_\pp[\ll_p,\ll_a]\]
	for any $b<a<p<f<e$ and $d<c$ (see the upper picture in Figure~\ref{fig:order});
	\item when $\pp$ is a lower marked point, say $\pp=\rr_p$, we have
	\[[\rr_p,\rr_f]>_\pp[\rr_p,\rr_e]>_\pp[\rr_p,\ll_d]>_\pp[\rr_p,\ll_c]>_\pp[\rr_p,\rr_b]>_\pp[\rr_p,\rr_a]\]
	for any $b<a<p<f<e$ and $d<c$ (see the lower picture in Figure~\ref{fig:order}).
\end{itemize}
Note that this order is not compatible with that given in \cite{LP}.

\begin{figure}[ht]\centering
	\begin{tikzpicture}[scale=.5]
	\draw[very thick] (-7,0)--(8,0);
	\draw[very thick] (-7,5)--(8,5);
	\draw (0,5)node{$\bullet$}node[above]{$\ll_p$} (-2,5)node{$\bullet$}node[above]{$\ll_a$} (-5,5)node{$\bullet$}node[above]{$\ll_b$}
	(-3,0)node{$\bullet$}node[below]{$\rr_c$}
	(3,0)node{$\bullet$}node[below]{$\rr_d$}
	(2,5)node{$\bullet$}node[above]{$\ll_f$} (5,5)node{$\bullet$}node[above]{$\ll_e$};
	\draw(-2,5).. controls +(-45:1) and +(-135:1)..(0,5);
	\draw(-5,5).. controls +(-60:2) and +(-120:2)..(0,5);
	\draw(0,5).. controls +(-110:2) and +(80:1)..(-3,0);
	\draw(0,5).. controls +(-70:2) and +(100:1)..(3,0);
	\draw(0,5).. controls +(-45:1) and +(-135:1)..(2,5);
	\draw(0,5).. controls +(-60:2) and +(-120:2)..(5,5);
	\end{tikzpicture}
	
	\begin{tikzpicture}[scale=.5]
	\draw[very thick] (-7,0)--(8,0);
	\draw[very thick] (-7,5)--(8,5);
	\draw (0,0)node{$\bullet$}node[below]{$\rr_p$} (-2,0)node{$\bullet$}node[below]{$\rr_f$} (-5,0)node{$\bullet$}node[below]{$\rr_e$}
	(-3,5)node{$\bullet$}node[above]{$\ll_d$}
	(3,5)node{$\bullet$}node[above]{$\ll_c$}
	(2,0)node{$\bullet$}node[below]{$\rr_a$} (5,0)node{$\bullet$}node[below]{$\rr_b$};
	\draw(-2,0).. controls +(45:1) and +(135:1)..(0,0);
	\draw(-5,0).. controls +(60:2) and +(120:2)..(0,0);
	\draw(0,0).. controls +(110:2) and +(-80:1)..(-3,5);
	\draw(0,0).. controls +(70:2) and +(-100:1)..(3,5);
	\draw(0,0).. controls +(45:1) and +(135:1)..(2,0);
	\draw(0,0).. controls +(60:2) and +(120:2)..(5,0);
	\end{tikzpicture}
	\caption{Linear order on the set $[\pp,-]$}
	\label{fig:order}
\end{figure}
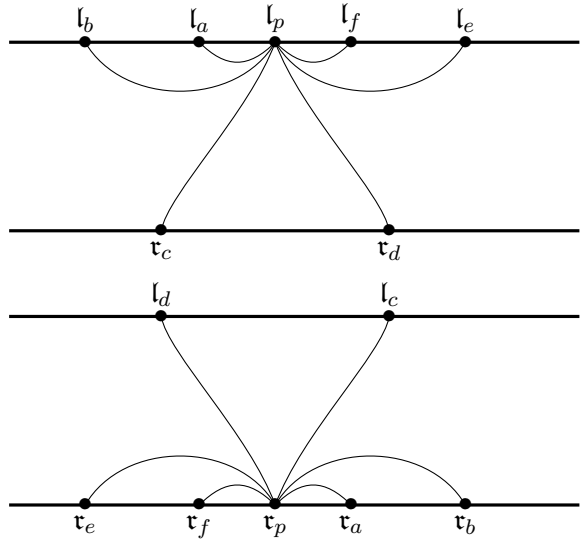

\subsection{Definition of compact Ptolemy diagrams}

The following definition of a Ptolemy diagram is an analogue of that in \cite{HJR1}.

\begin{definition}
A set $\TT$ of arcs in $\BB$ is called a Ptolemy diagram of $\BB$ if the following condition holds.
\begin{itemize}
   \item [$\cona$] For any two crossing arcs $[\mathfrak{p}, \mathfrak{q}]$ and $[\mathfrak{i}, \mathfrak{j}]$ in $\TT$, those of $[\mathfrak{p}, \mathfrak{i}]$, $[\mathfrak{p}, \mathfrak{j}]$, $[\mathfrak{q}, \mathfrak{i}]$, $[\mathfrak{q}, \mathfrak{j}]$ which are arcs are in $\TT$. (See Figure~\ref{fig:c}.)
 \end{itemize}
\end{definition}

\begin{figure}[ht]\centering
	
	\begin{tikzpicture}[xscale=.35,yscale=.4]
	\draw[very thick] (-7,0)--(8,0);
	\draw[very thick] (-7,5)--(8,5);
	\draw (-1,5)node{$\bullet$}node[above]{$\ii$} (-4,5)node{$\bullet$}node[above]{$\pp$}
	(5,5)node{$\bullet$}node[above]{$\jj$}
	(2,5)node{$\bullet$}node[above]{$\qq$}
	(2,0)node[below]{\color{white} $k$};
	\draw(-4,5).. controls +(-30:3) and +(-150:3)..(2,5);
	\draw(-1,5).. controls +(-30:3) and +(-150:3)..(5,5);
	\draw[dashed] (-4,5).. controls +(-25:1.3) and +(-165:1) .. (-1,5);
	\draw[dashed] (2,5).. controls +(-25:1.3) and +(-165:1) .. (5,5);
	\draw[dashed] (-1,5).. controls +(-25:1) and +(-165:1) .. (2,5);
	\draw[dashed] (-4,5).. controls +(-55:3) and +(-125:3) .. (5,5);
	\end{tikzpicture}
	\begin{tikzpicture}[xscale=.35,yscale=-.4]
	\draw[very thick] (-7,0)--(8,0);
	\draw[very thick] (-7,5)--(8,5);
	\draw (-1,5)node{$\bullet$}node[below]{$\ii$} (-4,5)node{$\bullet$}node[below]{$\pp$}
	(5,5)node{$\bullet$}node[below]{$\jj$}
	(2,5)node{$\bullet$}node[below]{$\qq$}
	(2,0)node[above]{\color{white} $k$};
	\draw(-4,5).. controls +(-30:3) and +(-150:3)..(2,5);
	\draw(-1,5).. controls +(-30:3) and +(-150:3)..(5,5);
	\draw[dashed] (-4,5).. controls +(-25:1.3) and +(-165:1) .. (-1,5);
	\draw[dashed] (2,5).. controls +(-25:1.3) and +(-165:1) .. (5,5);
	\draw[dashed] (-1,5).. controls +(-25:1) and +(-165:1) .. (2,5);
	\draw[dashed] (-4,5).. controls +(-55:3) and +(-125:3) .. (5,5);
	\end{tikzpicture}
	
	\begin{tikzpicture}[xscale=.35,yscale=.4]
	\draw[very thick] (-7,0)--(8,0);
	\draw[very thick] (-7,5)--(8,5);
	\draw (0,5)node{$\bullet$}node[above]{$\ii$} (-3,5)node{$\bullet$}node[above]{$\pp$}
	(-2,0)node{$\bullet$}node[below]{$\jj$}
	(3,5)node{$\bullet$}node[above]{$\qq$};
	\draw(-3,5).. controls +(-30:3) and +(-150:3)..(3,5);
	\draw(0,5).. controls +(-170:2) and +(80:1)..(-2,0);
	\draw[dashed] (-3,5).. controls +(-25:1.3) and +(-165:1) .. (0,5);
	\draw[dashed] (0,5).. controls +(-25:1.3) and +(-165:1) .. (3,5);
	\draw[dashed] (-3,5).. controls +(-60:2) and +(90:1) .. (-2,0);
	\draw[dashed] (-2,0).. controls +(60:1.3) and +(-120:3) .. (3,5);
	\end{tikzpicture}
	\begin{tikzpicture}[xscale=.35,yscale=-.4]
	\draw[very thick] (-7,0)--(8,0);
	\draw[very thick] (-7,5)--(8,5);
	\draw (0,5)node{$\bullet$}node[below]{$\qq$} (-3,5)node{$\bullet$}node[below]{$\ii$}
	(-2,0)node{$\bullet$}node[above]{$\pp$}
	(3,5)node{$\bullet$}node[below]{$\jj$};
	\draw(-3,5).. controls +(-30:3) and +(-150:3)..(3,5);
	\draw(0,5).. controls +(-170:2) and +(80:1)..(-2,0);
	\draw[dashed] (-3,5).. controls +(-25:1.3) and +(-165:1) .. (0,5);
	\draw[dashed] (0,5).. controls +(-25:1.3) and +(-165:1) .. (3,5);
	\draw[dashed] (-3,5).. controls +(-60:2) and +(90:1) .. (-2,0);
	\draw[dashed] (-2,0).. controls +(60:1.3) and +(-120:3) .. (3,5);
	\end{tikzpicture}
	
	\begin{tikzpicture}[xscale=.35,yscale=.4]
	\draw[very thick] (-7,0)--(8,0);
	\draw[very thick] (-7,5)--(8,5);
	\draw (1,0)node{$\bullet$}node[below]{$\ii$} (-3,5)node{$\bullet$}node[above]{$\pp$}
	(-2,0)node{$\bullet$}node[below]{$\jj$}
	(3,5)node{$\bullet$}node[above]{$\qq$};
	\draw[dashed] (-3,5).. controls +(-30:3) and +(-150:3)..(3,5);
	\draw[dashed] (1,0).. controls +(150:1) and +(50:1)..(-2,0);
	\draw (-3,5).. controls +(-70:2) and +(120:1) .. (1,0);
	\draw[dashed] (1,0).. controls +(90:1) and +(-100:1) .. (3,5);
	\draw[dashed] (-3,5).. controls +(-90:2) and +(90:1) .. (-2,0);
	\draw (-2,0).. controls +(60:1.3) and +(-120:3) .. (3,5);
	\end{tikzpicture}
	\caption{Condition $\cona$}
	\label{fig:c}
\end{figure}

For any set $\TT$ of arcs in $\cal B_{\infty}$, denote by \[\nc\TT=\{u\mid\text{$u$ does not cross any arcs in }\TT\}.\]
A large class of Ptolemy diagrams can be obtained in the following way.

\begin{lemma}\label{lem:9}
For any set $\TT$ of arcs in $\BB$, the set $\nc\TT$ is a Ptolemy diagram.
\end{lemma}

\begin{proof}
Let $[\pp,\qq]$ and $[\ii,\jj]$ be two crossing arcs in $\nc\TT$. We shall prove that if $[\pp,\ii]$ is an arc in $\BB$, then it is in $\nc\TT$. Assume conversely that there is an arc $u$ in $\TT$ crossing $[\pp,\ii]$. Then $u$ crosses either $[\pp,\qq]$ or $[\ii,\jj]$, a contradiction.
\end{proof}

Let $u_1$ and $u_2$ be crossing arcs in $\BB$. An arc or an edge $u_3$ is called a \emph{middle term from $u_2$ to $u_1$} if $u_2<_{\pp_1}u_3<_{\pp_2}u_1$ for some marked points $\pp_1$ and $\pp_2$ in $\BB$. It is easy to see that there are exactly two middle arcs from $u_2$ to $u_1$ and they are a pair of opposite sides of the quadrangle whose diagonals are $u_1$ and $u_2$. See Figure~\ref{fig:e}.

\begin{figure}[ht]\centering
	\begin{tikzpicture}[xscale=.35,yscale=.4]
	\draw (3,0)node{$\bullet$} (-3,5)node{$\bullet$}
	(-2,0)node{$\bullet$}
	(3,5)node{$\bullet$};
	\draw (-.3,4.5)node{$u_1$}
	(-.2,0.2)node{$u_2$};
	\draw (-3,5).. controls +(-20:3) and +(120:3)..(3,0);
	\draw (3,5).. controls +(150:1) and +(20:2)..(-2,0);
	\draw[dashed] (-3,5).. controls +(-90:2) and +(90:1) .. (-2,0);
	\draw[dashed] (3,5).. controls +(-90:2) and +(90:1) .. (3,0);
	\end{tikzpicture}
	\caption{Middle terms from $u_2$ to $u_1$}
	\label{fig:e}
\end{figure}
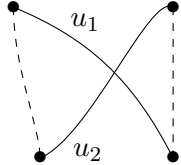

\begin{definition}
Let $\Omega$ be a set of arcs in $\BB$.
\begin{enumerate}
	\item A subset $\Sigma$ of $\Omega$ is called a {\em $\tau$-basis}  if for any arc $u_1\in\Omega$, there is an arc $u_2\in\Sigma$ such that $\tau u_2$ crosses $u_1$ and any middle term from $u_2$ to $u_1$ is in $\Omega$ when $u_2$ crosses $u_1$.
	\item A subset $\Sigma$ of $\Omega$ is called a {\em $\tau^{-1}$-basis}  if for any arc $u_1\in\Omega$, there is an arc $u_2\in\Sigma$ such that $\tau^{-1} u_2$ crosses $u_1$ and any middle term from $u_1$ to $u_2$ is in $\Omega$ when $u_2$ crosses $u_1$.
\end{enumerate}
\end{definition}

The following easy lemma is helpful for understanding the notion of $\tau$-basis.

\begin{lemma}\label{lem:4}
	Let $u_1$ and $u_2$ be two arcs in $\BB$ with $\tau u_2$ crossing $u_1$. Then precisely one of the following situations occurs:
	\begin{itemize}
		\item $u_1\geq_\pp u_2$ for some marked point $\pp$,
		\item $u_1$ crosses $u_2$.
	\end{itemize}
\end{lemma}

Let $\TT$ be a set of arcs in $\BB$. For each arc $u$ in $\BB$, denote by $\TT_u$ the subset of $\TT$ consisting of the arcs crossing $u$.

\begin{definition}
 A set $\TT$ of arcs in $\BB$ is called $\tau$-compact (resp. $\tau^{-1}$-compact) if $\TT_u$ admits a finite $\tau$-basis (resp. $\tau^{-1}$-basis) for every arc $u$ in $\BB$. A set $\TT$ of arcs in $\BB$ is called compact if it is both $\tau$-compact and $\tau^{-1}$-compact.
\end{definition}

\subsection{A criterion for a Ptolemy diagram to be compact}

This subsection is devoted to showing the following criterion for a Ptolemy diagram of $\BB$ to be $\tau$-compact.

\begin{theorem}\label{thm:cri}
A Ptolemy diagram $\TT$ of $\BB$ is $\tau$-compact if and only if $\TT$ satisfies the following conditions.
\begin{itemize}
\item[$\conb$]  Any marked point which is lower right $\TT$-bounded is upper right $\TT$-bounded, and any marked point which is upper left $\TT$-bounded is lower left $\TT$-bounded.
\item[$\conc$] $\TT\cup\nc\TT$ contains connecting arcs.
\end{itemize}
\end{theorem}

We also give the dual of $\conb$ as follows.
\begin{itemize}
\item[$\conbp$] Any marked point which is lower left $\TT$-bounded is upper left $\TT$-bounded, and any marked point which is upper right $\TT$-bounded is lower right $\TT$-bounded.
\end{itemize}
One can prove dually that a Ptolemy diagram of $\BB$ is $\tau^{-1}$-compact if and only if it satisfies $\conbp$ and $\conc$.

\begin{example}
Using Theorem~\ref{thm:cri}, we show the following Ptolemy diagrams to be $\tau$-compact or $\tau^{-1}$-compact.
\begin{enumerate}
	\item Any Ptolemy diagram of $\BB$ consisting of finitely many arcs is $\tau$-compact, e.g. $\TT_1=\{[\ll_{-2},\ll_1],[\ll_{-2},\rr_2],[\ll_{-2},\rr_{-2}],[\ll_1,\rr_2],[\ll_1,\rr_{-2}],[\rr_2,\rr_{-2}]\}$ shown in (1) of Figure~\ref{fig:d}. This is because, any marked point in $\BB$ is $\TT_1$-bounded in any of the four directions and $\nc\TT_1$ contains connecting arcs.
	\item $\TT_2=\TT_1\cup\{[\ll_1,\ll_p]\mid p\geq 3 \}$, see (2) in Figure~\ref{fig:d}. Note that $\ll_1$ is lower right $\TT_2$-bounded but not upper right $\TT_2$-bounded. Hence $\TT_2$ is not $\tau$-compact. (But $\TT_2$ is $\tau^{-1}$-compact.)
	\item $\TT_3=\TT_1\cup\{[\ll_1,\rr_q]\mid q\leq -5 \}$, see (3) in Figure~\ref{fig:d}. All marked points in $\BB$ are upper right $\TT_3$-bounded and lower left $\TT_3$-bounded. This, together with $\nc\TT_3$ containing connecting arcs, implies that $\TT_3$ is $\tau$-compact. (But $\TT_3$ is not $\tau^{-1}$-compact since it does not satisfy $\conbp$.)
\end{enumerate}
\end{example}

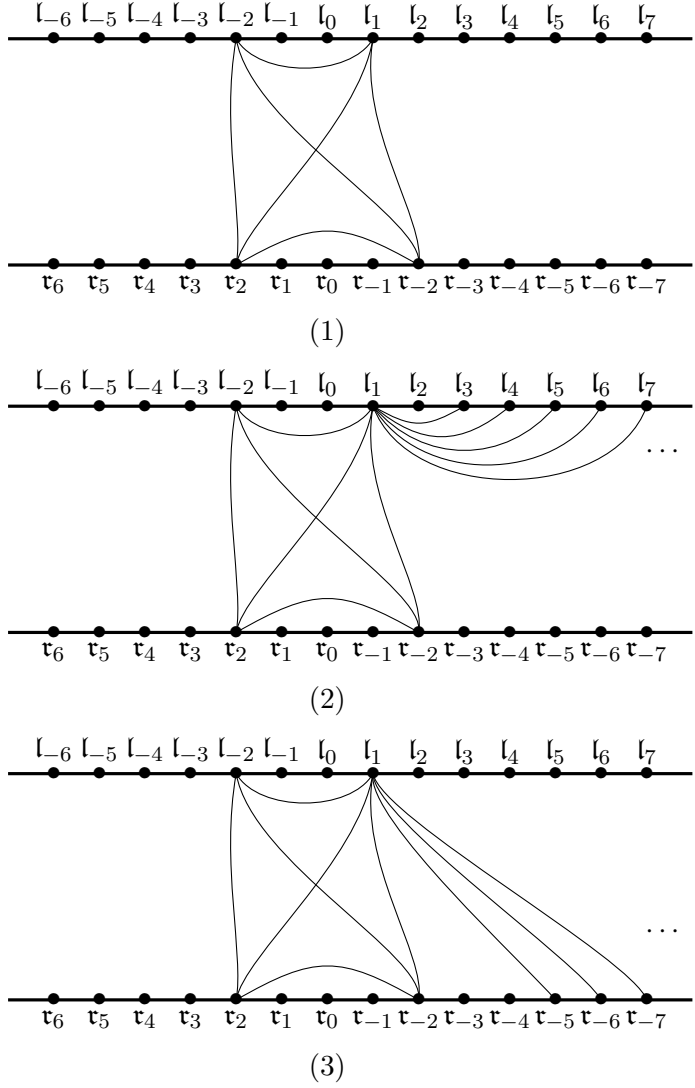
\begin{figure}[ht]\centering
	\begin{tikzpicture}[scale=.6]
	\draw[very thick] (-7,0)--(8,0);
	\draw[very thick] (-7,5)--(8,5);
	\foreach \x in {-6,-5,-4,...,7}
	\draw (\x,0)node{$\bullet$} (\x,5)node{$\bullet$}
	(\x,5)node[above]{$\ll_{\x}$};
	\foreach \x in {-7,-6,-5,...,6}
	\draw (-\x,0)node[below]{$\rr_{\x}$};
	\draw(-2,5).. controls +(-60:1) and +(-120:1)..(1,5);
	\draw(-2,5).. controls +(-100:2) and +(80:1)..(-2,0);
	\draw(1,5).. controls +(-100:2) and +(80:1)..(-2,0);
	\draw(-2,5).. controls +(-80:2) and +(80:1)..(2,0);
	\draw(-2,0).. controls +(30:2) and +(150:2)..(2,0);
	\draw(1,5).. controls +(-100:2) and +(80:1)..(2,0);
	\draw(0,-1.5)node{(1)};
	\end{tikzpicture}
	\begin{tikzpicture}[scale=.6]
	\draw[very thick] (-7,0)--(8,0);
	\draw[very thick] (-7,5)--(8,5);
	\foreach \x in {-6,-5,-4,...,7}
	\draw (\x,0)node{$\bullet$} (\x,5)node{$\bullet$}
	(\x,5)node[above]{$\ll_{\x}$};
	\foreach \x in {-7,-6,-5,...,6}
	\draw (-\x,0)node[below]{$\rr_{\x}$};
	\draw(-2,5).. controls +(-60:1) and +(-120:1)..(1,5);
	\draw(-2,5).. controls +(-100:2) and +(80:1)..(-2,0);
	\draw(1,5).. controls +(-100:2) and +(80:1)..(-2,0);
	\draw(-2,5).. controls +(-80:2) and +(80:1)..(2,0);
	\draw(-2,0).. controls +(30:2) and +(150:2)..(2,0);
	\draw(1,5).. controls +(-100:2) and +(80:1)..(2,0);
	\draw(0,-1.5)node{(2)};
	\draw(1,5).. controls +(-30:1) and +(-150:1)..(3,5);
	\draw(1,5).. controls +(-40:1.4) and +(-140:1.4)..(4,5);
	\draw(1,5).. controls +(-50:1.7) and +(-130:1.7)..(5,5);
	\draw(1,5).. controls +(-60:2) and +(-120:2)..(6,5);
	\draw(1,5).. controls +(-70:2.3) and +(-110:2.3)..(7,5);
	\draw(7.4,4)node{$\cdots$};
	\end{tikzpicture}
	\begin{tikzpicture}[scale=.6]
	\draw[very thick] (-7,0)--(8,0);
	\draw[very thick] (-7,5)--(8,5);
	\foreach \x in {-6,-5,-4,...,7}
	\draw (\x,0)node{$\bullet$} (\x,5)node{$\bullet$}
	(\x,5)node[above]{$\ll_{\x}$};
	\foreach \x in {-7,-6,-5,...,6}
	\draw (-\x,0)node[below]{$\rr_{\x}$};
	\draw(-2,5).. controls +(-60:1) and +(-120:1)..(1,5);
	\draw(-2,5).. controls +(-100:2) and +(80:1)..(-2,0);
	\draw(1,5).. controls +(-100:2) and +(80:1)..(-2,0);
	\draw(-2,5).. controls +(-80:2) and +(80:1)..(2,0);
	\draw(-2,0).. controls +(30:2) and +(150:2)..(2,0);
	\draw(1,5).. controls +(-100:2) and +(80:1)..(2,0);
	\draw(0,-1.5)node{(3)};
	\draw(1,5).. controls +(-100:1) and +(130:1)..(5,0);
	\draw(1,5).. controls +(-90:1) and +(120:1)..(6,0);
	\draw(1,5).. controls +(-80:1) and +(110:1)..(7,0);
	\draw(7.4,1.5)node{$\cdots$};
	\end{tikzpicture}
	\caption{Ptolemy diagrams of $\BB$}\label{fig:d}
\end{figure}

To prove Theorem~\ref{thm:cri}, we shall need some lemmas.

\begin{lemma}\label{lem:easy}
	Let $\TT$ be a set of arcs in $\BB$, satisfies $\conb$. Then $\TT_u$ also satisfies $\conb$, for every arc $u$ in $\BB$.
\end{lemma}

\begin{proof}
	We shall only prove that if an upper marked point $\ll_p$ is lower right $\TT_u$-bounded, then it is upper right $\TT_u$-bounded.
	\begin{itemize}
		\item If $\ll_p$ is lower right $\TT$-bounded, then by $\conb$ for $\TT$, $\ll_p$ is upper right $\TT$-bounded. In particular, $\ll_p$ is upper right $\TT_u$-bounded.
		\item If $\ll_p$ is not lower right $\TT$-bounded, then $u$ is neither a connecting arc with an endpoint $\ll_q$ with $q>p$ nor an upper arc $[\ll_i,\ll_j]$ with $i<p<j$. Hence there are only finitely many arcs $[\ll_p,\ll_t]$ in $\TT$, with $t>p$ and crossing $u$. So $\ll_p$ is upper right $\TT_u$-bounded.
	\end{itemize}
\end{proof}

\begin{lemma}\label{lem:min}
Let $\TT$ be a set of arcs in $\BB$, satisfying $\conb$. For every arc $u$ and every marked point $\pp$, if $\TT_u\cap [\pp,-]$ is nonempty, then it contains a (unique) minimal element.
\end{lemma}

\def\LL{\TT_u\cap [\pp,-]}

\begin{proof}
We shall only consider the case that $\pp$ is an upper marked point, say $\pp=\ll_p$. By Lemma~\ref{lem:easy}, $\TT_u$ satisfies $\conb$.
\begin{itemize}
	\item If there is an upper arc $[\ll_p,\ll_a]\in\TT_u$ with $a<p$, we may take $a$ to be maximal with respective to this property. Then $[\ll_p,\ll_a]$ is the minimal element in $\LL$.
	\item If there are no such upper arcs, then $\ll_p$ is upper left $\TT_u$-bounded. By $\conb$, $\ll_p$ is lower left $\TT_u$-bounded.
	\begin{itemize}
		\item If there are connecting arcs in $\LL$, then there is a maximal integer $c$ such that $[\ll_p,\rr_c]\in\TT_u$. Then $[\ll_p,\rr_c]$ is the minimal element in $\LL$.
		\item If there are no connecting arcs in $\LL$, then $\ll_p$ is lower right $\TT_u$-bounded. By $\conb$, $\ll_p$ is upper right $\TT_u$-bounded. So there is a maximal integer $e>p$ such that $[\ll_p,\ll_e]\in\TT_u$. Then $[\ll_p,\ll_e]$ is the minimal element in $\LL$.
	\end{itemize}
	
\end{itemize}

\end{proof}

We are ready to prove that $\conb$ and $\conc$ form a sufficient condition for a Ptolemy diagram to be $\tau$-compact.


\begin{proof}[Proof of the 'if' part of Theorem~\ref{thm:cri}]
Suppose both $\conb$ and $\conc$ hold for $\TT$. We need to prove $\TT_u$ admits a finite $\tau$-basis, for every arc $u$ in $\BB$. Consider first the case that $u$ is an upper arc in $\BB$, say $u=[\ll_p,\ll_q]$ with $p<q$. Then any arc crossing $ u$ has an endpoint $\ll_m$ with $p<m<q$. So $\TT_{u}=\cup_{p<m<q}(\TT_{u}\cap[\ll_m,-])$. By Lemma~\ref{lem:min}, each nonempty $\TT_{u}\cap[\ll_m,-]$ contains a unique minimal element. Then all of such minimal elements form a finite basis of $\TT_u$.

The case $u$ is a lower arc is similar. Consider now the case that $u$ is a connecting arc, say $u=[\ll_p,\rr_q]$.
If there is a connecting arc in $\nc\TT$, say $[\ll_i,\rr_j]$, then any arc in $\TT_u$ has an endpoint, which is either an upper marked point between $\ll_p$ and $\ll_i$, or a lower marked point between $\rr_q$ and $\rr_j$. So there is a finite set $S$ of marked points in $\BB$ such that $\TT_u=\cup_{\pp\in S}\TT_u\cap[\pp,-]$. Then the set of minimal elements in $\TT_{u}\cap[\pp,-]$ for $\pp\in S$ is a finite $\tau$-basis of $\TT_u$.

If there are no connecting arcs in $\nc\TT$, then by $\conc$ there is a connecting arc in $\TT$. Note that $\TT_u=A_1\cup A_2$, where $A_1=\cup_{p'>p}\TT_u\cap[\ll_{p'},-]$ and $A_2=\cup_{q'<q}\TT_u\cap[\rr_{q'},-]$. It suffices to find finite subsets $Z_{i}\subseteq A_{i}$, $i=1,2$, such that $Z_{i}$ is a $\tau$-basis of $A_i$  (since then $Z_{1}\cup Z_{2}$ is a $\tau$-basis of $\TT_u$). We shall only find $Z_1$. By Lemma~\ref{lem:min}, each nonempty $\TT_u\cap[\ll_{p'},-]$ has a minimal element. We denote by $\qq(p')$ the other endpoint of the minimal element. Since $[\ll_{p'},\qq(p')]$ crosses $u$, we have $\qq(p')$ is an upper marked point left to $\ll_p$ or a lower marked point left to $\rr_q$. We claim that for any $p''>p'>p$, $[\ll_{p'},\qq(p')]$ and $[\ll_{p''},\qq(p'')]$ do not cross each other. Indeed, if they cross then by $\cona$ we have $[\ll_{p'},\qq(p'')]\in\TT_u$, which is smaller than $[\ll_{p'},\qq(p')]$, a contradiction. Let $Z_{1}$ be the set of minimal elements in $\TT_u\cap[\qq(p'),-]$ for $p'>p$. We claim that $Z_1\cup Z_2$ is a $\tau$-basis of $\TT_u$. In fact, suppose $u_1$ is an arc in $\TT_u$. Without lose of generality, we suppose $u_1=[\ll_{p'},\aa]$ is in $A_1$ for some marked point $\aa$. We have $[\ll_{p'},\qq(p')]\in\TT_u\cap[\qq(p'),-]$. Suppose the minimal element in $\TT_u\cap[\qq(p'),-]$ with $p'>p$ is $u_2=[\qq(p'),\qqq(p')]$ ), i.e, $u_2$ is in $Z_1$. If $\qq(p')=\aa$ or $\ll_{p'}=\qqq(p')$, then $u_1$ does not cross $u_2$, and $u_1$ crosses $\tau u_2$. So we consider the case $\qq(p')\neq\aa$ and $\ll_{p'}\neq\qqq(p')$. Then $u_1$ crosses $u_2$, $[\ll_{p'},\qq(p')]$ and $[\aa,\qqq(p')]$ are two middle terms from $u_2$ to $u_1$. Obviously, they are in $\TT_u$. Next we show that the set $\{\qq(p')\mid p'>p\}$ is finite, which implies that $Z_1$ is finite and we are done.
\begin{itemize}
	\item [(1)] If there is an integer $p'>p$ such that $\qq(p')$ is a lower marked point, then for any $p''>p'$, $\qq(p'')$ is a lower marked point between $\qq(p')$ and $\rr_q$. This is because $[\ll_{p'},\qq(p')]$ and $[\ll_{p''},\qq(p'')]$ do not cross each other.  Hence, $\{\qq(p')\mid p'>p\}$ is finite.
	\item [(2)] Consider now the case that all of $\qq(p')$ are upper marked points. Note that there exists an arc $v=[\ll_i,\rr_j]\in\TT$. If $v$ does not cross any upper arcs $[\ll_{p'},\qq(p')]$ with $p'>p$, then any $\qq(p')$ is between $\ll_i$ and $\ll_p$ and hence $\{\qq(p')\mid p'>p\}$ is finite. If $v$ crosses an upper arc $[\ll_{p'},\qq(p')]$ for some $p'>p$, by $\cona$, we have $[\qq(p'),\rr_j]\in\TT$ for some $p'>p$. Then by $\cona$ again, we have $[\ll_{p''},\qq(p')]\in\TT_u$ for any $p''>p'$, which implies that $\qq(p'')=\qq(p')$. So the set $\{\qq(p')\mid p'>p\}$ is also finite and we complete the proof the claim.
\end{itemize}

\end{proof}

The next result shows that $\conb$ is a necessary condition for a set of arcs (not necessarily a Ptolemy diagram)  to be $\tau$-compact.

\begin{proposition}\label{lem:5}
Any $\tau$-compact set $\TT$ of arcs in $\BB$ satisfies condition $\conb$.
\end{proposition}
\begin{proof}
We shall only show that any upper marked point which is lower right $\TT$-bounded is upper right $\TT$-bounded, and any upper marked point which is upper left $\TT$-bounded is lower left $\TT$-bounded.
\begin{itemize}
  \item [(1)] Let $\ll_p$ be an upper marked point which is lower right $\TT$-bounded. Then there is an integer $s$ such that $[\ll_p,\rr_{s'}]\notin\TT$ for any $s'<s$. Let $u=[\ll_{p+1},\rr_s]$. So $\TT_u\cap[\ll_p,-]$ consists of the arcs $[\ll_p,\ll_t]\in\TT$ with $t>p+1$. If $\ll_p$ is not upper right $\TT$-bounded, then $\TT_u\cap[\ll_p,-]$ is infinite. Since $\TT$ is $\tau$-compact, $\TT_u$ admits a finite $\tau$-basis $\Sigma$. By the finiteness of $\Sigma$, there is an arc $u_1\in\TT_u\cap[\ll_p,-]$, which is smaller than any arc in $\Sigma\cap[\ll_p,-]$. On the other hand, denoting by $\ll_q$ the other endpoint of $u_1$, any arc in $[\ll_q,-]$ which is smaller than $u_1$ does not crosses $u$. Hence by Lemma~\ref{lem:4}, for any arc $u_2\in\Sigma$ satisfying $\tau u_2$ crosses $u_1$, we have that $u_1$ crosses $u_2$. So $u_2$ has an endpoint of the form $\ll_r$ with $p<r<q$. Then the arc $[\ll_r,\ll_q]$ is a middle term from $u_2$ to $u_1$. However, $[\ll_r,\ll_q]$ does not cross $u$. This contradicts that $\Sigma$ is a $\tau$-basis of $\TT_u$.

  \item [(2)] Let $\ll_p$ be an upper marked point which is upper left $\TT$-bounded. Let $s$ be the minimal integer such that $[\ll_p,\ll_s]$ is in $\TT$ or is an edge and let $u=[\ll_s,\rr_t]$ an arbitrary connecting arc having $\ll_s$ as an endpoint. So $\TT_u\cap[\ll_p,-]$ consists of the connecting arcs $[\ll_p,\rr_{t'}]\in\TT$ with $t'>t$.
  If $\ll_p$ is not lower left $\TT$-bounded, then $\TT_u\cap[\ll_p,-]$ is infinite. Since $\TT$ is $\tau$-compact, $\TT_u$ admits a finite $\tau$-basis  $\Sigma$. By the finiteness of $\Sigma$, there is an integer $m$ such that $\rr_n$ is not an endpoint of any arc in $\Sigma$ for any $n>m$. Then by the infiniteness of $\TT_u\cap[\ll_p,-]$, there exists a connecting arc $u_1=[\ll_p,\rr_q]\in\TT_u\cap[\ll_p,-]$ with $q>\max\{m,t\}$. It follows that any arc in $[\rr_q,-]$ or $[\ll_p,-]$, which is smaller than $u_1$, is not in $\Sigma$. Then by Lemma~\ref{lem:4} for any arc $u_2\in\Sigma$ with $\tau u_2$ crossing $u_1$, $u_1$ crosses $u_2$. Using the fact that any lower marked point left to $\rr_q$ is not an endpoint of any arc in $\Sigma$, we have that $u_2$ has an endpoint, which is an upper marked point $\ll_r$ with $r<p$. Then $[\ll_r,\ll_p]$ is a middle term from $u_2$ to $u_1$. But if $r\geq s$, $[\ll_r,\ll_p]$ does not cross $u$; if $r<s$, $[\ll_r,\ll_p]$ is not in $\TT$ by the minimality of $s$. Therefore $[\ll_r,\ll_p]\notin\TT_u$, which contradicts that $\Sigma$ is a $\tau$-basis of $\TT_u$.

\end{itemize}
\end{proof}

An upper marked point $\ll_p$ is said to be covered by an upper arc $[\ll_i,\ll_j]$ if $i<p<j$; and a lower marked point $\rr_q$ is said to be covered by a lower arc $[\rr_a,\rr_b]$ if $a>q>b$.
Now we can complete the proof of Theorem~\ref{thm:cri} by the following result.

\begin{proposition}\label{prop:com}
Any $\tau$-compact Ptolemy diagram $\TT$ of $\BB$ satisfies $\conc$.
\end{proposition}

\begin{proof}
	
To show that $\TT$ satisfies $\conc$ is equivalent to proving that if $\nc\TT$ does not contain any connecting arcs, then $\TT$ contains connecting arcs. Then either every upper marked point is covered by an upper arc in $\TT$, or every lower marked point is covered by a lower arc in $\TT$. Without loss of generality, we assume that the former occurs. Let $\ll_p$ be an upper marked point and $u$ an arbitrary connecting arc having $\ll_p$ as an endpoint. Since $\TT$ is $\tau$-compact, $\TT_u$ admits a finite basis $\Sigma$. By the finiteness of $\Sigma$, there are integers $m<p<n$ such that for any $m'< m$ and $n'> n$ there are no upper arcs in $\Sigma$ which has $\ll_{m'}$ or $\ll_{n'}$ as an endpoint. Assume that there exists an upper arc $[\ll_i,\ll_j]\in\TT$ with $i<p<j$ such that $\min\{|p-i|,|p-j|\}$ is maximal. Since $\ll_i$ is an upper marked point, there is an upper arc $[\ll_a,\ll_b]\in\TT$ with $a<i<b$. By the maximality of $\min\{|p-i|,|p-j|\}$, we have that $b<j$. By $\cona$ we have $[\ll_a,\ll_j]$ is in $\TT$. However $[\ll_a,\ll_j]$ covers $\ll_p$ and $\min\{|p-a|,|p-j|\}>\min\{|p-i|,|p-j|\}$, a contradiction. Hence there are integers $m'< m$ and $n'> n$ such that $[\ll_{m'},\ll_{n'}]\in\TT$ covering $\ll_p$. So $[\ll_{m'},\ll_{n'}]$ is in $\TT_u$ and it does not cross any upper arcs in $\tau\Sigma$. Therefore, there are connecting arcs in $\Sigma$. As $\Sigma$ is a subset of $\TT$, it follows that $\TT$ contains connecting arcs.
\end{proof}

\subsection{A crucial property of compact Ptolemy diagrams}

For any set $\TT$ of arcs in $\BB$, denote by $\bTT$ the set obtained from $\TT$ by adding all edges in $\BB$.

\begin{lemma}\label{a}
	Let $\TT$ be a Ptolemy diagram and let $\pp$ be a marked point in $\BB$. For any two elements $[\pp,\ii]>_\pp[\pp,\jj]$ in $\bTT\cap[\pp,-]$, if there is no $[\pp,\qq]$ in $\bTT\cap[\pp,-]$ with $[\pp,\ii]>_\pp[\pp,\qq]>_\pp[\pp,\jj]$, then $[\ii,\jj]$ is in $\overline{\nc\TT}$.
\end{lemma}

\begin{proof}
	If $[\ii,\jj]$ is an arc but not in $\nc\TT$, then there is an arc $u$ in $\TT$ crossing $[\ii,\jj]$. Note that the arc $[\ii,\jj]$ divides the infinite strip $\BB$ into two regions. Let $\qq$ be the endpoint of $u$ that is in the different region from $\pp$. It follows that $[\pp,\ii]>_\pp [\pp,\qq]>_\pp[\pp,\jj]$. So $\pp$ is not an endpoint of $u$. But this implies that $u$ crosses one of $[\pp,\ii]$ and $[\pp,\jj]$. By $\cona$, we have $[\pp,\qq]\in\TT$, a contradiction.
\end{proof}

Let $u$ be an arc in $\BB$ and $\pp$ an endpoint of $u$. Denote by $[\pp,-]_{>_\pp u}$ the subset of $[\pp,-]$ consisting of the elements bigger than $u$.

\begin{lemma}\label{lem:relmin}
	Let $\TT$ be a set of arcs in $\BB$, satisfying $\conb$. Then there is a minimal element in $\bTT\cap[\pp,-]_{>_\pp u}$ for any arc $u$ in $\BB$, where $\pp$ is an endpoint of $u$.
\end{lemma}

\begin{proof}
	If there is not a minimal element in $[\pp,-]_{>_\pp u}$, then one of the following situations occurs.
	\begin{enumerate}
		\item $\pp$ is upper left $\TT$-bounded but is not lower left $\TT$-bounded.
		\item $\pp$ is lower right $\TT$-bounded but is not upper right $\TT$-bounded.
	\end{enumerate}
	This contradicts condition~$\conb$.
\end{proof}

We shall need the following lemma.

\begin{lemma}\label{lem:8}
Let $\TT$ be a $\tau$-compact Ptolemy diagram of $\BB$. If there is a connecting arc $u$ in $\nc\TT$, which is not in $\TT$, then there is another connecting arc in $\nc\TT$ crossing $u$.
\end{lemma}

\begin{proof}
Let $\ll_p$ and $\rr_q$ be the two endpoints of $u$. Using Lemma~\ref{lem:relmin}, there are minimal elements $[\ll_p,\pp]$ and $[\rr_q,\qq]$ in $\bTT\cap[\ll_p,-]_{> u}$ and $\bTT\cap[\rr_q,-]_{> u}$, respectively. It is clear that the arc $[\pp,\qq]$ crosses $u$. To complete the proof, we only need to prove $[\pp,\qq]$ is in $\nc\TT$. Indeed, if there is an arc $v=[\ii,\jj]\in\TT$ crossing $[\pp,\qq]$, then neither $\ll_p$ nor $\rr_q$ is an endpoint of $v$ by $u\notin \TT$ and the minimality of $[\ll_p,\pp]$ and $[\rr_q,\qq]$. So $[\ii,\jj]$ crosses either $[\ll_p,\pp]$ or $[\rr_q,\qq]$. Without loss of generality, we assume that $[\ii,\jj]$ crosses $[\ll_p,\pp]$. Since $[\ii,\jj]$ does not cross $u$, both $[\ll_p,\ii]$ and $[\ll_p,\jj]$ are bigger than $u$. By $\cona$, both $[\ll_p,\ii]$ and $[\ll_p,\jj]$ are in $\TT$. But we have either $[\ll_p,\ii]>_{\ll_p}[\ll_p,\pp]>_{\ll_p}[\ll_p,\jj]$ or $[\ll_p,\jj]>_{\ll_p}[\ll_p,\pp]>_{\ll_p}[\ll_p,\ii]$. Both cases contradict the minimality of $[\ll_p,\pp]$.
\end{proof}

We have the following important property of a $\tau$-compact Ptolemy diagram.

\begin{proposition}\label{prop:1}
Any $\tau$-compact Ptolemy diagram $\TT$ satisfies $\TT = \nc\nc\TT$.
\end{proposition}
\begin{proof}
	
The inclusion $\TT\subseteq\nc\nc\TT$ is clear. So it suffices to show that any arc $u$ in $\nc\nc\TT$ is in $\TT$.
By Theorem~\ref{thm:cri}, $\TT$ satisfies $\conb$ and $\conc$.

Consider first the case that $u$ is an upper arc, say $u=[\ll_p, \ll_q]$ with $p<q$. Let $s>p$ be the minimal integer such that $[\ll_q,\ll_s]\in\bTT$. So $s\leq q-1$. By Lemma~\ref{lem:relmin} there is a minimal element in $\bTT\cap[\ll_q,-]_{>_{\ll_q} [\ll_q,\ll_s]}$, say $[\ll_q,\pp]$. By the minimality of $s$, $\pp$ is not an upper marked point $\ll_t$ with $p< t\leq q$. Hence if $\pp\neq\ll_p$ then $[\pp,\ll_s]$ crosses $u$. This is a contradiction because by Lemma~\ref{a} $[\pp,\ll_s]\in\nc\TT$. The proof in case $u$ is a lower arc is similar.
	
Consider now the last case that $u$ is a connecting arc, say $u=[\ll_p, \rr_q]$. We claim that $\TT$ contains a connecting arc. Indeed, if $\TT$ does not contain any connecting arc, then by $\conc$ there are connecting arcs in $\nc\TT$. Let $v=[\ll_i,\rr_j]$ be a connecting arc in $\nc\TT$ such that $|p-i|+|q-j|$ is minimal. By Lemma~\ref{lem:8}, there is an arc $w\in\nc\TT$ crossing $v$. It follows that $v\neq u$ and one of the endpoints of $w$, say $\pp$, is either an upper marked point between $\ll_p$ and $\ll_i$ with not equaling $\ll_i$, or a lower marked point between $\rr_q$ and $\rr_j$ with not equaling $\rr_j$. By Lemma~\ref{lem:9}, $\nc\TT$ satisfies $\cona$. So both $[\ll_i,\pp]$ and $[\rr_j,\pp]$ are in $\nc\TT$. But one of them is a connecting arc, which is nearer to $u$ than $v$, a contradiction. Thus, there are connecting arcs in $\TT$.
	
Let $[\ll_m,\rr_n]$ be a connecting arc in $\TT$ with $|m-p|+|n-q|$ minimal. We need to prove $|m-p|+|n-q|=0$.
\begin{enumerate}
	\item If $n<q$, then there is a minimal integer $r$ such that $q>r\geq n$ and $[\rr_q,\rr_r]\in\bTT$. Using Lemma~\ref{lem:relmin} there is a minimal element $[\rr_q,\pp]$ in $\TT\cap[\rr_q,-]_{>_{\rr_q}[\rr_q,\rr_r]}$. By the minimality of $r$, we have $[\rr_q,\pp]>_{\rr_q}[\rr_q,\rr_n]$. By Lemma~\ref{a}, $[\rr_r,\pp]$ is in $\overline{\nc\TT}$. It follows that $[\rr_r,\pp]$ does not cross $[\ll_p,\rr_q]$. Hence $\pp$ is either an upper marked point $\ll_a$ with $a\geq p$ or a lower marked point $\rr_b$ with $b<n$. If $[\rr_q,\pp]$ crosses $[\ll_m,\rr_n]$, then by $\cona$, we have $[\ll_m,\rr_q]\in\TT$ with $|m-p|+|q-q|<|m-p|+|n-q|$, a contradiction. If $[\rr_q,\pp]$ does not cross $[\ll_m,\rr_n]$, then $\pp=\ll_a$ with $a\leq m$. It follows that we have $[\ll_a,\rr_q]\in\TT$ with $|a-p|+|q-q|<|m-p|+|n-q|$, a contradiction.
	\item The case $m<p$ can be proved similarly as (1).
	\item If $n>q$ and $m\geq p$, by Lemma~\ref{lem:relmin}, there is a minimal element $[\ll_m,\pp]$ in $\TT\cap[\ll_m,-]_{>[\ll_m,\rr_n]}$. Then by Lemma~\ref{a}, $[\rr_n,\pp]$ is in $\nc\TT$. It follows that $\pp=\rr_b$ for some $n>b\geq q$. So we have $[\ll_m,\rr_b]\in\TT$ with $|m-p|+|b-q|<|m-p|+|n-q|$, a contradiction.
	\item The case that $n\geq q$ and $m>p$ can be proved similarly as (3).
\end{enumerate}
\end{proof}

\section{Geometric realization of cotorsion pairs}\label{sec:gr}


We shall use the following lemmas to prove the main result in the paper.

\begin{lemma}\label{lem:easy2}
For any two arcs $u_1$ and $u_2$ in $\BB$ sharing an endpoint $\pp$, we have that $u_2\geq_\pp u_1$ if and only if $\Hom(M_{u_1},M_{u_2})\neq 0.$
\end{lemma}

\begin{proof}
It is easy to see that $u_2\geq_\pp u_1$ if and only if $u_2$ crosses $\tau u_1$. By Lemma~\ref{lem1}, the lemma follows.
\end{proof}

\begin{lemma}\label{lem:fac}
Let $v_1$ and $v_2$ be two arcs in $\BB$ with $v_2\geq_\pp v_1$, for some marked point $\pp$. Then for any arc $u$ with $\tau u$ crossing $v_1$, any morphism from $M_u$ to $M_{v_2}$ factors through an arbitrary morphism from $M_{v_1}$ to $M_{v_2}$.
\end{lemma}

\begin{proof}
Let $M=M_u$, $N=M_{v_1}$ and $L=M_{v_2}$. Since $\tau u$ crosses $v_1$ and $\tau v_1$ crosses $v_2$, there are nonzero morphisms $f:M\to N$ and $g:N\to L$. Moreover, we have $\Hom(N,L[1])=0$ because $v_1$ does not cross $v_2$. Using the dual of \cite[Lemma~5.6]{LP}, $\Hom(M_u,M_{v_2})$ is generated by $gf$. In particular, any map from $M_u$ to $M_{v_2}$ factors through an arbitrary morphism from $M_{v_1}$ to $M_{v_2}$.
\end{proof}

\begin{lemma}\label{lem:3}
Let $\Y$ be a subcategory of $\CC$ and $u$ an arc in $\BB$. Then there is a left $\Y$-approximation of $M_u$ if and only if $\YY_{\tau u}$ admits a finite $\tau$-basis.
\end{lemma}

\begin{proof}

To prove the `if' part, let $\Sigma$ be a finite $\tau$-basis of $\YY_{\tau u}$. For each arc $v\in\Sigma$, there is a non-zero morphism $f_v:M_u\to M_v$ since $v$ crosses $\tau u$. We claim that $f=\oplus_{v\in\Sigma}f_v:M_u\to \oplus_{v\in\Sigma}M_v$ is a left $\Y$-approximation of $M_u$. Indeed, for any non-zero morphism $g:M_u\to M_{u_1}$ where $u_1$ is an arc in $\YY$, we have $\tau u$ crosses $u_1$. So $u_1\in\YY_{\tau u}$ and hence there is an arc $u_2\in\Sigma$ such that $\tau u_2$ crosses $u_1$ and any middle term from $u_2$ to $u_1$ is in $\YY_{\tau u}$. It suffices to prove that $g$ factors through $f_{u_2}$. By Lemma~\ref{lem:4}, there are the following two cases.
\begin{enumerate}
	\item $u_1\geq_{\pp} u_2$ for some marked point $\pp$. By Lemma~\ref{lem:fac}, we have that $g$ factors through $f_{u_2}$.
	\item $u_1$ crosses $u_2$. Then by definition, there exists a middle term $u_3$ with $u_2<_{\pp_1}u_3<_{\pp_2}u_1$ for some marked point $\pp_1$ and $\pp_2$. Since $u_3\in\YY_{\tau u}$, using Lemma~\ref{lem:fac} repeatedly, we have that $g$ factors through $f_{u_2}$.
\end{enumerate}

We now prove the `only if' part. Let $f:M_u\to X$ be a minimal left $\Y$-approximation. We may write $X=\oplus_{v\in\Sigma}M_v$, where $\Sigma$ is a finite set of arcs in $\BB$. It follows that $\Sigma$ is a subset of $\YY_{\tau u}$. Let $u_1$ be an arc in $\YY_{\tau u}$. Then for any nonzero morphism $g:M_u\to M_{u_1}$, there is a morphism $h:\oplus_{v\in\Sigma}M_v\to M_{u_1}$ such that $g=h\circ f$. Denote by $f_v:M_u\to M_v$ (resp. $h_v:M_v\to M_{u_1}$) the restricting of $f$ (resp. $h$) to $M_v$. Since $h\circ f\neq 0$, we have $f_{u_2}\neq 0$ and $h_{u_2}\neq 0$ for some $u_2\in\Sigma$. So by Proposition~\ref{lem1} $\tau u_2$ crosses $u_1$. Let $u_3$ be a middle term from $u_2$ to $u_1$. We need to show $u_3\in\YY_{\tau u}$, which completes the proof. Indeed, if $u_3$ does not cross $\tau u$, then $\Hom(M_u,M_{u_3})=0$. But on the other hand, by Lemma~\ref{lem:easy2} $\Hom(M_{u_2},M_{u_3})\neq0$ and $\Hom(M_{u_3},M_{u_1})\neq 0$. By Lemma~\ref{lem:fac} it follows that $h_{u_2}:M_{u_2}\to M_{u_1}$ factors through $M_{u_3}$. So $h_{u_2}\circ f_{u_2}=0$, which is a contradiction.
\end{proof}



The main result of this paper is the following classification of cotorsion pairs in $\CC$ by compact Ptolemy diagrams of $\BB$.

\begin{theorem}\label{maintheorem}
Let $\X$, $\Y$ be subcategories of $\mathscr{C}$, and $\widetilde{\X}$ and $\widetilde{\Y}$ the corresponding sets of arcs in $\BB$, respectively. Then the following statements are equivalent.
\begin{enumerate}
\item $(\X,\Y)$ is a cotorsion pair in $\mathscr{C}$.
\item $\widetilde{\X}$ is a $\tau^{-1}$-compact Ptolemy diagram of $\BB$ and $\widetilde{\Y}=\nc\widetilde{\X}$.
\item $\widetilde{\Y}$ is a $\tau$-compact Ptolemy diagram of $\BB$ and  $\widetilde{\X}=\nc\widetilde{\Y}$.
\item $\widetilde{\X}$ satisfies conditions $\cona$, $\conbp$ and $\conc$ and $\widetilde{\Y}=\nc\widetilde{\X}$.
\item $\widetilde{\Y}$ satisfies conditions $\cona$, $\conb$ and $\conc$ and $\widetilde{\X}=\nc\widetilde{\Y}$.
\end{enumerate}
\end{theorem}

\begin{proof}
The equivalences between (1), (3) and (5) follows directly from Proposition~\ref{prop:def}, Proposition~\ref{prop:1}, Theorem~\ref{thm:cri} and Lemma~\ref{lem:3}. The other equivalences can be proved dually.
\end{proof}

\section{Applications}

\subsection{Classification of functorially finite rigid subcategories and cluster tilting subcategories in \texorpdfstring{$\CC$}{C}}

A {\em partial triangulation} of $\BB$ is a collection of non-crossing arcs in $\BB$; and a triangulation of $\BB$ is a maximal collection of non-crossing arcs in $\BB$. Clearly, any (partial) triangulation satisfies $\cona$, and hence it is a Ptolemy diagram.

\begin{definition}[Definition~4.9 and Definition~4.11 in \cite{LP}]
A (partial) triangulation $\TT$ of $\BB$ is called compact if for every arc $u\in\BB$, $\TT_u$ admits a finite subset $\Sigma$ such that every arc in $\TT_u$ crosses some arc of $\tau\Sigma$ as well as some arc of $\tau^{-1}\Sigma$.
\end{definition}

This compactness is compatible with ours in the following sense.

\begin{lemma}\label{lem:n1}
	A (partial) triangulation of $\BB$ is compact if and only if it is both $\tau$-compact and $\tau^{-1}$-compact as a Ptolemy diagram.
\end{lemma}

\begin{proof}
This follows directly from Lemma~\ref{lem:4}.
\end{proof}

By Proposition~\ref{prop:def}, $\X$ is functorially finite rigid if and only if $(\X, \X[-1]^\bot)$ and $(^\bot\X[1],\X)$ are cotorsion pairs. Thus, we have the following classification of functorially finite rigid subcategories of $\CC$ .

\begin{proposition}
	Let $\X$ be a subcategory of  $\mathscr{C}$. Then the following statements are equivalent.
	\begin{enumerate}
		\item The subcategory $\X$ is functorially finite rigid.
		\item $\XX$ is a compact partial triangulation of $\BB$.
		\item $\XX$ is a partial triangulation satisfying $\conb$, $\conbp$ and $\conc$.
	\end{enumerate}
\end{proposition}

\begin{proof}
By Proposition~\ref{lem1}, for a subcategory $\X$ of $\CC$, $\X$ is rigid if and only if $\XX$ is a partial triangulation of $\BB$. Then this proposition follows by Theorem~\ref{maintheorem} and Lemma~\ref{lem:n1}
\end{proof}

As a direct consequence, we classify cluster tilting subcategories; compare \cite[Theorem 5.7]{LP}. Note that by \cite{ZZ}, the cluster tilting categories of $\CC$ are the functorially finite maximal rigid subcategories. 




\begin{corollary}[Theorem 5.7 in \cite{LP}]
	Let $\X$ be a subcategory of $\mathscr{C}$. Then the following statements are equivalent.
	\begin{enumerate}
		\item The subcategory $\X$ is cluster tilting.
		\item $\XX$ is a compact triangulation of $\BB$.
        \item $\XX$ is a triangulation of $\BB$ containing connecting arcs, and every marked point in $\BB$ which is upper left (resp. right) $\XX$-bounded is also lower left (resp. right) $\XX$-bounded and vice versa.
	\end{enumerate}
\end{corollary}
\begin{proof}
The only fact we should point out is that (3) is equivalent to that $\XX$ is a triangulation satisfying $\conb$, $\conbp$ and $\conc$.

\end{proof}

\subsection{Classification of \texorpdfstring{$t$}{t}-structures in \texorpdfstring{$\CC$}{C}}

Recall that a $t$-structure is a cotorsion pair $(\X,\Y)$ such that $\X$ is closed under $[1]$ or equivalently $\Y$ is closed under $[-1]$.

For any integer $p$, denote by $L_{\leq p}$ (resp. $R_{\leq p}$) the set of upper arcs $[\ll_i,\ll_j]$ (resp. lower arcs $[\rr_i,\rr_j]$) in $\BB$ with $i,j\leq p$; denote by $L_{\geq p}$ (resp. $R_{\geq p}$) the set of upper arcs $[\ll_i,\ll_j]$ (resp. lower arcs $[\rr_i,\rr_j]$) in $\BB$ with $i,j\geq p$. For convenience, take $L_{\leq -\infty}$, $R_{\leq -\infty}$, $L_{\geq +\infty}$ and $R_{\geq +\infty}$ to be the empty set. We now give a classification of $t$-structures in $\CC$ as a application of our main result.
	
\begin{theorem}\label{t-structures}
Let $\X$ and $\Y$ be subcategories of $\CC$ and $\XX$ and $\YY$ the corresponding sets of arcs in $\BB$. Then $(\X,\Y)$ is a $t$-structure in $\CC$ if and only if
\begin{enumerate}
\item $\XX=L_{\geq p}\cup R_{\geq q}$ for $p,q$ integers or $-\infty$ and $\YY=\nc\XX$; or
\item $\YY=L_{\leq p}\cup R_{\leq q}$ for $p,q$ integers or $+\infty$ and $\XX=\nc\YY$.
\end{enumerate}
Moreover, in each case, the heart of $(\X,\Y)$ is the subcategory of $\CC$ corresponding to the set  $\{[\ll_{p-1},\ll_{p+1}],[\rr_{q-1},\rr_{q+1}]\}$, where $[\ll_{p-1},\ll_{p+1}],[\rr_{q-1},\rr_{q+1}]$ need to be omitted if $p,q$ are not integers, respectively.
\end{theorem}
	
\begin{proof}
To prove the `only if' part, suppose that $(\X,\Y)$ is a $t$-structure in $\CC$. By Theorem~\ref{maintheorem}, $\XX=\nc\YY$, $\YY=\nc\XX$, $\XX$ satisfies $\cona$, $\conbp$ and $\conc$, and $\YY$ satisfies $\cona$, $\conb$ and $\conc$. By $\conc$, there is a connecting arc in $\XX\cup\YY$. Consider first the case that $\XX$ contains connecting arcs. Let $p$ (resp. $q$) be the minimal integer such that $\ll_p$ (resp. $\rr_q$) is an endpoint of some arc in $\XX$ if exists, or $p=-\infty$ ($q=-\infty$) otherwise. Since $\X$ is closed under $[1]$, by Proposition~\ref{lem1}, $\XX$ is closed under $\tau$. Using the action of $\tau$ and $\cona$ repeatedly, we have that any marked point $\ll_i$ (resp. $\rr_j$) with $i\geq p$ (resp. $j\geq q$) is an endpoint of some connecting arc in $\widetilde{\X}$. It follows that there are no connecting arcs in $\YY$. Therefore, by the minimality of $p$ and $q$ and by $\YY=\nc\XX$, we have  $\YY=L_{\leq p}\cup R_{\leq q}$. Thus, we show that if $\XX$ contains connecting arcs, then (2) holds. Similarly, we can prove that if $\YY$ contains connecting arcs, then (1) holds.

To show the `if' part, it is easy to see that $\XX$ in case (1) satisfies $\cona$, $\conbp$ and $\conc$ and that $\YY$ in case (2) satisfies $\cona$, $\conb$ and $\conc$. Hence by Theorem~\ref{maintheorem}, $(\X,\Y)$ is a cotorsion pair. Moreover, in both cases, $\XX\cap\YY=\emptyset$. Then by Proposition~\ref{prop:def}, $(\X,\Y)$ is a $t$-structure.

\end{proof}

We illustrate the two types of $t$-structures in Theorem~\ref{t-structures} in Figure~\ref{hh1} and Figure~\ref{hh2}, respectively. The diagram in Figure~\ref{hh1} corresponds to the left part of a $t$-structure and the diagram in Figure~\ref{hh2} corresponds to the right part of a $t$-structure. Note that when $p$ or $q$ is $+\infty$ or $-\infty$, some arcs in the figures will disappear.

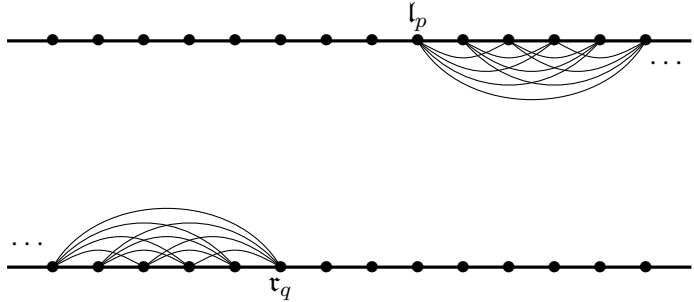
\begin{figure}[ht]\centering
	\begin{tikzpicture}[scale=.6]
	\draw[very thick] (-7,0)--(8,0);
	\draw[very thick] (-7,5)--(8,5);
	\foreach \x in {-6,-5,-4,...,7}
	\draw (\x,0)node{$\bullet$} (\x,5)node{$\bullet$};
	\draw (2,5)node[above]{$\ll_p$};
	\draw (-1,0)node[below]{$\rr_q$};
	\draw(2,5).. controls +(-30:1) and +(-150:1)..(4,5);
	\draw(2,5).. controls +(-40:1.4) and +(-140:1.4)..(5,5);
	\draw(2,5).. controls +(-50:1.7) and +(-130:1.7)..(6,5);
	\draw(2,5).. controls +(-60:2) and +(-120:2)..(7,5);
	\draw(3,5).. controls +(-30:1) and +(-150:1)..(5,5);
	\draw(3,5).. controls +(-40:1.4) and +(-140:1.4)..(6,5);
	\draw(3,5).. controls +(-50:1.7) and +(-130:1.7)..(7,5);
	\draw(4,5).. controls +(-30:1) and +(-150:1)..(6,5);
	\draw(4,5).. controls +(-40:1.4) and +(-140:1.4)..(7,5);
	\draw(5,5).. controls +(-30:1) and +(-150:1)..(7,5);
	\draw (7.5,4.5)node{$\cdots$};
	
	\draw(-1,0).. controls +(180-30:1) and +(180-150:1)..(-3,0);
	\draw(-1,0).. controls +(180-40:1.4) and +(180-140:1.4)..(-4,0);
	\draw(-1,0).. controls +(180-50:1.7) and +(180-130:1.7)..(-5,0);
	\draw(-1,0).. controls +(180-60:2) and +(180-120:2)..(-6,0);
	\draw(-2,0).. controls +(180-30:1) and +(180-150:1)..(-4,0);
	\draw(-2,0).. controls +(180-40:1.4) and +(180-140:1.4)..(-5,0);
	\draw(-2,0).. controls +(180-50:1.7) and +(180-130:1.7)..(-6,0);
	\draw(-3,0).. controls +(180-30:1) and +(180-150:1)..(-5,0);
	\draw(-3,0).. controls +(180-40:1.4) and +(180-140:1.4)..(-6,0);
	\draw(-4,0).. controls +(180-30:1) and +(180-150:1)..(-6,0);
	\draw(-6.5,.5)node{$\cdots$};
	\end{tikzpicture}
	\caption{The first type of $t$-structures}
	\label{hh1}
\end{figure}

\begin{figure}[ht]\centering
	\begin{tikzpicture}[scale=.6]
	\draw[very thick] (-7,0)--(8,0);
	\draw[very thick] (-7,5)--(8,5);
	\foreach \x in {-6,-5,-4,...,7}
	\draw (\x,0)node{$\bullet$} (\x,5)node{$\bullet$};
	\draw (2-3,5)node[above]{$\ll_p$};
	\draw (-1+3,0)node[below]{$\rr_q$};
	\draw(2-8,5).. controls +(-30:1) and +(-150:1)..(4-8,5);
	\draw(2-8,5).. controls +(-40:1.4) and +(-140:1.4)..(5-8,5);
	\draw(2-8,5).. controls +(-50:1.7) and +(-130:1.7)..(6-8,5);
	\draw(2-8,5).. controls +(-60:2) and +(-120:2)..(7-8,5);
	\draw(3-8,5).. controls +(-30:1) and +(-150:1)..(5-8,5);
	\draw(3-8,5).. controls +(-40:1.4) and +(-140:1.4)..(6-8,5);
	\draw(3-8,5).. controls +(-50:1.7) and +(-130:1.7)..(7-8,5);
	\draw(4-8,5).. controls +(-30:1) and +(-150:1)..(6-8,5);
	\draw(4-8,5).. controls +(-40:1.4) and +(-140:1.4)..(7-8,5);
	\draw(5-8,5).. controls +(-30:1) and +(-150:1)..(7-8,5);
	\draw (-6.5,4.5)node{$\cdots$};
	
	\draw(-1+8,0).. controls +(180-30:1) and +(180-150:1)..(-3+8,0);
	\draw(-1+8,0).. controls +(180-40:1.4) and +(180-140:1.4)..(-4+8,0);
	\draw(-1+8,0).. controls +(180-50:1.7) and +(180-130:1.7)..(-5+8,0);
	\draw(-1+8,0).. controls +(180-60:2) and +(180-120:2)..(-6+8,0);
	\draw(-2+8,0).. controls +(180-30:1) and +(180-150:1)..(-4+8,0);
	\draw(-2+8,0).. controls +(180-40:1.4) and +(180-140:1.4)..(-5+8,0);
	\draw(-2+8,0).. controls +(180-50:1.7) and +(180-130:1.7)..(-6+8,0);
	\draw(-3+8,0).. controls +(180-30:1) and +(180-150:1)..(-5+8,0);
	\draw(-3+8,0).. controls +(180-40:1.4) and +(180-140:1.4)..(-6+8,0);
	\draw(-4+8,0).. controls +(180-30:1) and +(180-150:1)..(-6+8,0);
	\draw(7.5,.5)node{$\cdots$};
	\end{tikzpicture}
	\caption{The second type of $t$-structures}
	\label{hh2}
\end{figure}
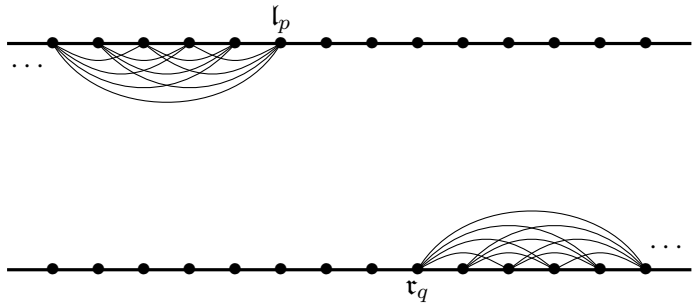

\begin{remark}
By Theorem \ref{t-structures}, each pair $(p,q)$ gives two $t$-structures in $\CC$. Hence there is a bijection from $\left(\mathbb{Z}\cup\{\infty\}\right)\times\left(\mathbb{Z}\cup\{\infty\}\right)\times\mathbb{Z}_2$ to the set of t-structures in $\CC$.


\end{remark}

An immediate corollary of Theorem~\ref{t-structures} is the following.

\begin{corollary}
The heart of any non-trivial $t$-structure in $\mathscr{C}$ is equivalent to the module category of the algebra $\k$ or $\k\oplus\k$.
\end{corollary}

\subsection{Relationship with the cluster category of type \texorpdfstring{$A_\infty$}{A infinity}}

In this subsection, we use our classification of cotorsion pairs in $\CC$ to recover the main result in \cite{Ng} which gives a classification of cotorsion pairs in the cluster category of type $A_\infty$.

\def\UU{\mathcal{U}}

Let $u=[\ll_p,\rr_q]$ an arbitrary connecting arc in $\BB$. Set
\[\nc M_u=^\bot(\add M_u[1]).\]
Let $\CC_u$ be the quotient category $\nc M_u/[\add M_u]$, whose objects are the same as the objects of $\nc M_u$ with morphisms given by the morphisms of $\nc M_u$ modulo those morphisms factoring through $\add M_u$. For any object $M$ of $\nc M_u$, denote by $\overline{M}$ the corresponding object of $\CC_u$. By \cite[Section~4]{IY}, $\CC_u$ is a 2-Calabi-Yau triangulated category and for any cluster tilting subcategory $\D$ of $\CC$ containing $M_u$, the subcategory of $\CC_u$ generated by the objects $\overline{M}$, $M\in\D$, is a cluster tilting subcategory of $\CC_u$.

By Proposition~\ref{lem1}, there is a bijection
\[\nc\{u\}\xrightarrow{1-1} \{\text{the (isoclasses of) indecomposable objects of $\nc M_u$}\}\]
sending $v$ to $M_v$. This induces a bijection
\begin{equation}\label{eq:bi}
\nc\{u\}\setminus\{u\}\xrightarrow{1-1} \{\text{the (isoclasses of) indecomposable objects of $\CC_u$}\}
\end{equation}
sending $v$ to $\overline{M_v}$. For any subcategory $\D$ of $\CC_u$, we use $\widetilde{\D}$ to denote the subset of $\nc\{u\}\setminus\{u\}$ consisting of $v$  with $\overline{M_v}\in\D$.

Let $\D_1$ and $\D_2$ be the subcategories of $\CC$ such that $\widetilde{\D_1}$ and $\widetilde{\D_2}$ consist of the arcs in $\nc\{u\}\setminus\{u\}$ left to $u$ and right to $u$, respectively. We have the following result.

\begin{theorem}\label{C}
Let $u$ be a connecting arc in $\BB$. Using the notation above, we have that $\D_1$ and $\D_2$ are triangulated subcategories of $\CC_u$ such that $\CC_u=\D_1\oplus\D_2$. Moreover, $\D_i$ are equivalent to the cluster category of type $A_\infty$.
\end{theorem}

\begin{proof}
By the bijection \eqref{eq:bi}, any indecomposable object in $\CC_u$ is either in $\D_1$ or in $\D_2$. On the other hand, for any two arcs $v_i\in\D_i$, we have
\[
\Ext^1_{\CC_u}(\overline{M_{v_1}},\overline{M_{v_2}})\cong\Ext^1_{\CC}(\overline{M_{v_1}},\overline{M_{v_2}})=0
\]
where the first isomorphism follows from \cite[Lemma~4.8]{IY} and the second one from Proposition~\ref{lem1}. Hence the first assertion of the theorem follows.

To show the second assertion, let $\T_i$ be the subcategory of $\D_i$ such that
\[
\widetilde{\T_1}=\{[\ll_{p-i},\rr_{q+i}],\ [\ll_{p-i},\rr_{q+i-1}]\mid i>0 \},
\]
\[
\widetilde{\T_2}=\{[\ll_{p-i},\rr_{q+i}],\ [\ll_{p-i-1},\rr_{q+i}]\mid i<0 \}.
\]
Then the union $\widetilde{\T_1}\cup\widetilde{\T_2}\cup\{u\}$ is a compact triangulation of $\BB$ (cf. the following figure). It follows that $\T_i$ is a cluster tilting subcategory of $\D_i$.

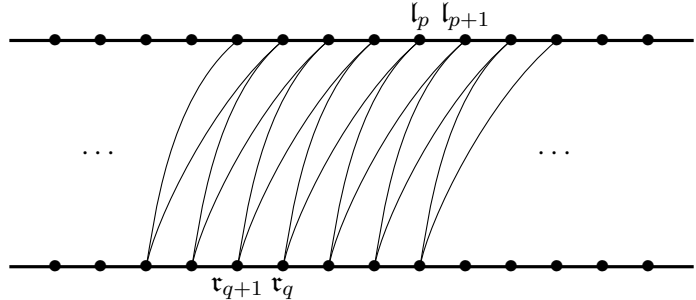
\centerline{\begin{tikzpicture}[scale=.6]
	\draw[very thick] (-7,0)--(8,0);
	\draw[very thick] (-7,5)--(8,5);
	\foreach \x in {-6,-5,-4,...,7}
	\draw (\x,0)node{$\bullet$} (\x,5)node{$\bullet$};
	\draw (2,5)node[above]{$\ll_p$} (3,5)node[above]{$\ll_{p+1}$};
	\draw (-1,0)node[below]{$\rr_q$} (-2,0)node[below]{$\rr_{q+1}$};
	\foreach \x in {-1,0,1,...,5}
	\draw(\x,5).. controls +(-140:2) and +(80:1)..(\x-3,0);
	\foreach \x in {-1,0,1,...,5}
	\draw(\x-1,5).. controls +(-140:2) and +(80:1)..(\x-3,0);
	\draw (5,2.5)node{$\cdots$} (-5,2.5)node{$\cdots$};
	\end{tikzpicture}}

Since $\T_i$ is acyclic of type $A_{\infty}$, by \cite[Theorem~3.2]{SR}, $\mod\T_i$ is hereditary. By \cite{KR2}, it follows that $\D_i$ is equivalent to the cluster category of type $A_\infty$.
\end{proof}

\def\UU{\mathcal{U}}

Let us recall some notion from \cite{Ng}. A set $\{m,n\}$ of two integers with $|n-m|\geq 2$ is called an arc in $L_\infty$. Let $V$ be the set of arcs in $L_\infty$. Two arcs $\{m_1,n_1\}$ and $\{m_2,n_2\}$ are said to cross if either $m_1<m_2<n_1<n_2$ or $m_2<m_1<n_2<n_1$. A set of arcs $\UU$ is said to satisfy {\em condition (i)} if, for each pair of crossing arcs $\{m_1,n_1\}$ and $\{m_2,n_2\}$ in $\UU$, those of the pairs $\{m_1,m_2\}$, $\{m_1,n_2\}$, $\{n_1,m_2\}$ and $\{n_1,n_2\}$ which are arcs belong to $\UU$. A set of arcs $\UU$ is said to satisfy {\em condition (f)} provided that for any integer $m$, if there are infinitely many arcs in $\UU$ of the form $\{m,n\}$ with $n>m$ then there are infinitely many arcs in $\UU$ of the form $\{m,n\}$ with $n<m$.

Consider the bijection $\varphi$ from the set $\{\ll_{p-i},\ \rr_{q+i}\mid i\geq 0 \}$ to the set of integers, sending $\ll_{p-i}$ to $i+1$ and sending $\rr_{q+i}$ to $-i$. Then $\varphi$ induces a bijection from the set $\widetilde{D_1}$ to $V$, sending $v=[\pp,\qq]$ to $\varphi(v):=[\varphi(\pp),\varphi(\qq)]$. This bijection, together with Theorem~\ref{C}, gives a one-to-one correspondence between the set $V$ and the set of (isoclasses of) indecomposable objects in $\C_{A_\infty}$. Hence for any subcategory $\X$ of $\C_{A_\infty}$, there is a corresponding subset $\XX$ of $V$. Then we have the following corollary of Theorem~\ref{maintheorem}.


\begin{corollary}[Theorem 3.18 in \cite{Ng}]
Let $\X$ be a subcategory of $\C_{A_\infty}$ and let $\XX$ be the corresponding subset of $V$. Then $(\X,\X^\bot)$ is a torsion pair if and only if $\XX$ satisfies condition (i) and condition (f).
\end{corollary}

\begin{proof}
It is straightforward to see that $\XX$ satisfies condition (i) if and only if $\varphi^{-1}(\XX)\cup\{u\}$ is a Ptolemy diagram of $\BB$. It is also easy to see that $\XX$ satisfies condition (f) if and only if $\varphi^{-1}(\XX)\cup\{u\}$ satisfies condition $\conbp$. Since $\conc$ always holds for $\varphi^{-1}(\XX)\cup\{u\}$, by Theorem~\ref{maintheorem}, we have that $\XX$ satisfies condition (i) and condition (f) if and only if $(\X',\X'[-1]^\bot)$ is a cotorsion pair in $\CC$, where $\X'$ is the subcategory of $\CC$ whose indecomposable object corresponds to an arc in $\varphi^{-1}(\XX)\cup\{u\}$. On the other hand, by \cite[Theorem 3.5]{ZZ3}, $(\X',\X'[-1]^\bot)$ is a cotorsion pair in $\CC$ if and only if $(\X,\X[-1]^\bot)$ is a cotorsion pair in $\CC_u$. Hence we are done.
\end{proof}

\begin{remark}
The subcategory of $\CC$, whose corresponding set of arcs in $\BB$ is the set of all upper arcs, is equivalent to the cluster category $\C_{A_\infty}$ of type $A_\infty$. There is clearly a canonical bijection between the set of upper arcs in $\BB$ and the set $V$. However, this bijection does not gives a one-to-one correspondence between the $\tau^{-1}$-compact Ptolemy diagrams of $\BB$ which only contains upper arcs and the subsets of $V$ satisfying condition (i) and condition (f). Hence one can not deduce Ng's classification of torsion pairs in $\C_{A_\infty}$ in this way.
\end{remark}

\end{document}